\documentclass[11pt, reqno]{amsart}
\usepackage{amscd,amsfonts}
\usepackage{amssymb}
\usepackage{amsmath}
\usepackage{amsthm}
\usepackage[all]{xy}
\theoremstyle{plain}
\newtheorem{thm}{Theorem}[subsection]

\newtheorem{lem}{Lemma}[subsection]
\newtheorem{prop}{Proposition}[subsection]

\newtheorem{conj}{Conjecture}[subsection]

\theoremstyle{definition}

\newcommand{\V}{\mathcal{V}}

\newcommand{\rank}{\textrm{rank}}

\newcommand{\Ext}{\operatorname{Ext}}
\newcommand{\Hom}{\operatorname{Hom}}
\newcommand{\Ann}{\operatorname{Ann}}
\newcommand{\MaxSpec}{\operatorname{MaxSpec}}
\newcommand{\Ker}{\textrm{Ker}}

\newcommand{\0}{\bar 0}
\renewcommand{\1}{\bar 1}

\newcommand{\HH}{\operatorname{H}}
\newcommand{\gl}{\mathfrak{gl}(n)}
\newcommand{\Stab}{\ensuremath{\operatorname{Stab}}}
\newcommand{\Norm}{\ensuremath{\operatorname{Norm}}}
\newcommand{\Lie}{\ensuremath{\operatorname{Lie}}}
\newcommand{\resstar}{\ensuremath{\operatorname{res}^{*}}}
\numberwithin{equation}{subsection}
\def\Z{{\mathbb Z}}

\def\:{\colon}
\def\a{\alpha}
\def\b{\mathfrak{b}}
\def\t{\mathfrak{t}}
\def\e{\mathfrak{e}}

\def\la{\lambda}

\def\La{\mathfrak{g}}
\def\Lta{\mathfrak{f}}
\def\Lsa{\mathfrak{h}}
\def\C{{\mathbb C}}

\newcommand{\fg}{\La}
\newcommand{\fe}{\e}
\newcommand{\ff}{\Lta}
\newcommand{\fh}{\Lsa}
\newcommand{\fa}{\mathfrak{a}}
\begin{document}
\title[Cohomology and Support Varieties for Lie Superalgebras of type $W(n)$]
{Cohomology and Support Varieties for Lie Superalgebras of type $W(n)$}

\author{Irfan Bagci }
\address{Department of Mathematics \\
            University of Georgia \\
            Athens, GA 30602}
\email{bagci@math.uga.edu}
\thanks{Research of the first author was partially supported by NSF grant DMS-0401431}
\author{Jonathan R. Kujawa}
\address{Department of Mathematics \\
            University of Oklahoma \\
            Norman, OK 73019}
\thanks{Research of the second author was partially supported by NSF grant
DMS-0734226}\
\email{kujawa@math.ou.edu}
\author{Daniel K. Nakano}
\address{Department of Mathematics \\
            University of Georgia \\
            Athens, GA 30602}
\thanks{Research of the third author was partially supported by NSF
grant  DMS-0654169}\
\email{nakano@math.uga.edu}

\date{\today}



\begin{abstract} Boe, Kujawa and Nakano \cite{BKN1, BKN2} recently investigated  
relative cohomology for classical Lie superalgebras and developed a theory 
of support varieties. The dimensions of these support varieties give a geometric
interpretation of the combinatorial notions of defect and atypicality 
due to Kac, Wakimoto, and Serganova. In this paper we calculate the cohomology 
ring of the Cartan type Lie superalgebra $W(n)$ relative to the degree zero component 
$W(n)_{0}$ and show that this ring is a finitely generated polynomial ring. This allows one 
to define support varieties for finite dimensional $W(n)$-supermodules which 
are completely reducible over $W(n)_{0}$. We calculate the support varieties 
of all simple supermodules in this category. Remarkably our computations coincide 
with the prior notion of atypicality for Cartan type superalgebras due to Serganova. 
We also present new results on the realizability of support varieties which 
hold for both classical and Cartan type superalgebras. 
\end{abstract}

\maketitle

\parskip=2pt

\section{Introduction} 

\subsection{} Let $\fg=\fg_{\0}\oplus \fg_{\bar{1}}$ be a finite dimensional 
simple Lie superalgebra over the complex numbers ${\mathbb C}$. 
In 1977 Kac provided a complete classification of these Lie superalgebras 
(cf. \cite{Kac}). The simple finite dimensional Lie superalgebras are divided into two types based 
on their degree $\0$ part: 
they are either \emph{classical} (when $\fg_{\0}$ is reductive) or \emph{of Cartan type} (otherwise). 
The Lie superalgebras of Cartan type consist of four infinite families of superalgebras: 
$W(n), S(n),\tilde{S}(n)$ and $H(n)$.

Let us first summarize what is known for the classical Lie superalgebras \cite{BKN1, BKN2}.  The first fundamental result is that the relative cohomology 
ring $\text{H}^\bullet(\La, \La_{\0};\C)$ is a finitely generated commutative ring.  Note that this result crucially depends on the 
reductivity of $\fg_{\0}$. By applying invariant theory 
results in \cite{LR} and \cite{dadokkac}, it was shown under mild conditions that a natural ``detecting'' subalgebra 
$\fe=\fe_{\0}\oplus \fe_{\bar{1}}$ of $\fg$ 
arises such that the restriction map in cohomology induces an isomorphism 
\[
R:=\text{H}^{\bullet}(\fg,\fg_{\0};{\mathbb C})
\cong \text{H}^{\bullet}(\fe,\fe_{\0};{\mathbb C})^{W},
\] 
where $W$ is a finite pseudoreflection group. The vector space dimension of the degree $\1$ part of the detecting subalgebra 
and the Krull dimension of $R$ both coincide with the combinatorial notion of the defect of $\fg$ previously 
introduced by Kac and Wakimoto \cite{kacwakimoto}. The fact that $R$ is finitely generated can be employed to 
define the cohomological support varieties ${\mathcal V}_{(\fe,\fe_{\0})}(M)$ and 
${\mathcal V}_{(\fg,\fg_{\0})}(M)$ for any finite dimensional $\fg$-supermodule $M$. 
The variety ${\mathcal V}_{(\fe,\fe_{\0})}(M)$ can be identified as a certain subvariety of 
$\fe_{\bar{1}}$ using a rank variety description \cite[Theorem 6.3.2]{BKN1}. For the Lie superalgebra $\fg=\mathfrak{gl}(m|n)$ 
the support varieties of all finite dimensional simple supermodules were computed in \cite{BKN2}. 
A remarkable consequence of this calculation is that the dimensions of the support varieties of a given simple 
supermodule (over $\fg$ or $\fe$) concides with the combinatorially defined degree of atypicality of the highest weight as defined by Kac and Serganova. 

\subsection{} In this paper we demonstrate that one can also use relative cohomology and 
support varieties in the setting of the Cartan type superalgebra $W(n)$. Recall that the Lie superalgebra $W(n)$ is the Lie 
superalgebra of superderivations of the exterior algebra $\Lambda (n)$ on $n$ generators.  Since $\Lambda(n)= \oplus_{k} \Lambda^{k}(n)$ has a natural $\Z$-grading given by total degree, one obtains a $\Z$-grading on the Lie superalgebra, $W(n)=\bigoplus _{i=-1}^{n-1}W(n)_i$, where $D \in W(n)$ is of degree $i$ if $D(\Lambda^{k}(n)) \subseteq \Lambda^{i+k}(n)$ for all $k \in \Z.$ The zero graded component $W(n)_0$ is isomorphic to $\gl$ and $W(n)\cong \Lambda (V)\otimes V^\ast$ as a $\gl$-module, where $V$ is the natural $n$ dimensional 
representation of $\gl$. 

The crucial difference between the Cartan type superalgebras and the classical superalgebras 
is that the $\fg_{\0}$ component is no longer reductive.  However, as was described above for $W(n),$ the Cartan type Lie algebras of types 
$W(n)$, $S(n)$, and $H(n)$ admit a ${\mathbb Z}$-grading: $\fg=\oplus_{i\in {\mathbb Z}} \fg_{i}.$  The grading is compatible with the $\Z_{2}$-grading in the sense that $\oplus_{i}\fg_{2i}=\fg_{\0}$ and $\oplus_{i}\fg_{2i+1}=\fg_{\1}$. Furthermore, the bracket respects the grading (i.e.\ $[\fg_{i},\fg_{j}] \subseteq \fg_{i+j}$ for all integers $i,j$).  In particular, $\fg_{0}$ is a reductive Lie algebra under this bracket.  It is then natural to consider the category of $\fg$-supermodules which are finitely semisimple over $\fg_{0}$.  All finite dimensional simple $\fg$-supermodules, for example, are objects in this category. 
Furthermore, as we will see, the reductiveness of $\fg_{0}$ implies the cohomology ring for this category is a finitely generated algebra. 

The paper is organized as follows. Set $(\La, \La_0)=(W(n),W(n)_0)$ and $G_0 \cong \text{GL}(n)$ to be the connected reductive group such that $\Lie (G_{0}) =\fg_{0}$ and such that the action of $G_{0}$ on $\fg$ differentiates to the adjoint action of $\fg_{0}$ on $\fg.$  In Sections~\ref{S:prelims}~--~\ref{S:OnWn} we briefly review basic facts on relative cohomology 
proved in \cite{BKN1} and on $W(n)$ from \cite{Ser} that will be needed for this paper. Section~\ref{S:CohominC}  is devoted to applying these results 
to the pair $(\La, \La_0)$. In particular, we use the representation theory of 
$\mathfrak{gl}(n)$ to show in Theorem~\ref{cohomologyring} that 
\[
R:=\HH^{\bullet}(\fg,\fg_{0};\C)
\]
can be identified with a ring of invariants and, consequently, is finitely generated. We also prove that when $M$ is a finite dimensional $\fg$-supermodule, $\HH^{\bullet}(\fg,\fg_{0};M)$ is a finitely generated $R$-module.  This proposition is the key first step to developing a theory of cohomological support varieties. In Section~\ref{S:invarianttheory} we invoke invariant theory results due to Luna and Richardson \cite{LR} 
to construct a detecting subsuperalgebra $\ff=\ff_{\0}\oplus \ff_{\bar{1}}$ for $\La$ so that the inclusion $\ff \hookrightarrow \fg$ induces the following isomorphism in cohomology,
$$R=\HH^\bullet(\La, \La_{0};\C)\cong \HH^\bullet(\ff,\ff_{\0};\C)^{\Sigma_{n-1}}$$
where $\Sigma_{n-1}$ is the symmetric group on $n-1$ letters. In particular this will imply $R$ is a polynomial ring in $n-1$ variables. We remark that there are some similarities in technique 
to Premet's \cite{Pr} earlier computation of the ring of invariant functions on the finite-dimensional 
simple Lie algebra $W(n,1)$ over an algebraically closed field of positive characteristic. 

Given a finite dimensional $\La$-supermodule, $M$, one can use the aforementioned finite generation results to define the 
cohomological support varieties $\V_{(\La, \La_0)}(M)$ and $\V_{(\ff, \ff_{\0})}(M)$. The properties 
of these varieties are described in Section~\ref{S:suppvar}. We also provide computations of these support varieties 
for all finite dimensional simple $\fg$-supermodules using the work of Serganova \cite{Ser}. 
Our results demonstrate that the dimension of these varieties geometrically realize the combinatorial notion of 
atypicality due to Serganova. Finally, in Section~\ref{S:realization} we prove a general realization theorem for 
conical subvarieties of the spectrum of $R$ which holds in both the classical case and for $W(n)$.  

We remark that Duflo and Serganova introduce associated varieties for finite dimensional supermodules of a Lie superalgebra in \cite{DS}.  Whether $\fg$ is classical or $W(n)$ it remains unclear what connection, if any, exists between their work and the cohomological support varieties considered in \cite{BKN1, BKN2} and here.

\subsection{Acknowledgements} The second author is grateful to the Mathematical Sciences Research Institute 
for their support and hospitality during the preparation of this manuscript. 

\section{Preliminaries}\label{S:prelims} 

\subsection{} For further details on Lie superalgebras and relative cohomology we refer the 
reader to \cite[\S2]{BKN1}. Throughout we work with the complex numbers $\C$ as the ground field. 
A finite dimensional Lie superalgebras is a $\Z_2$-graded vector space $\La=\La_{\0}\oplus \La_{\bar{1}}$ with a bracket
operation $[-,-] : \La\otimes \La \rightarrow  \La$ which preserves the $\Z_2$-grading and satisfies graded versions of the usual axioms for a Lie bracket.  In particular, note that $\La_{\0}$ is a Lie algebra upon restriction of the bracket.  If $\La$ is a Lie superalgebra, then we denote by $U(\fg )$ the universal enveloping superalgebra of $\fg$. As a special case of this setup we always view a Lie algebra as a Lie superalgebra concentrated in degree $\0.$

Given a Lie superalgebra $\fg$ the category of $\La$-supermodules has as objects the left $U(\La)$-modules 
which are $\Z_2$-graded and such that the action of $U(\fg )$ respects this grading.  By definition a subsupermodule of a supermodule, say $ N \subseteq M,$ satisfies $N_{r} = N \cap M_{r}$ for $r \in \Z_2.$   Morphisms are as described in \cite[Section 2.1]{BKN1}; in particular, we do not assume that morphisms preserve the $\Z_2$-grading. Consequently the category of $\La$-supermodules is not an abelian category.  However the graded (or underlying even category), consisting of the same objects but with only $\Z_{2}$-grading preserving morphisms, is 
an abelian category. This along with the parity change functor 
(which interchanges the $\Z_2$-grading of a supermodule) allows one to 
make use of the standard tools of homological algebra. Since $U(\fg)$ is a Hopf superalgebra one can use the antipode
(resp. coproduct) of $U(\La)$ to define a $\La$-supermodule structure on the linear dual of a supermodule 
denoted by  $M^\ast$ (resp. the tensor product of two supermodules over $\C$ denoted by  $M\otimes N$). As a matter of 
notation, for a homogeneous element $x$ in a $\Z_2$-graded vector space we write $\bar{x} \in \Z_{2}$ for the degree 
of the element.  We call $x$ \emph{even} if $\bar{x}=\0$ and \emph{odd} if $\bar{x}=\1.$

If $\fg$ is a Lie superalgebra, then a $\fg$-supermodule $M$ is 
\emph{finitely semisimple} if it is isomorphic to a direct sum of finite dimensional simple $\fg$-supermodules.
Let ${\mathfrak t}$ be a Lie subsuperalgebra of $\fg$ and let ${\mathcal C}=\mathcal{C}_{(\La, \t)}$ be the full subcategory of the category of all $\La$-supermodules which has its objects all $\fg$-supermodules which are finitely semisimple as $\t$-supermodules. The category $\mathcal{C}_{(\La, \t)}$ is closed under arbitrary direct sums, 
quotients, and finite tensor products (cf. \cite[3.1.6]{Kum}). Given $M$ and $N$ in $\mathcal{C}$, we write 
$\Ext_\mathcal{C}^d(M,N)$ for the degree $d$ extensions between $N$ and $M$ in $\mathcal{C}$.

\subsection{Relative Cohomology for Lie superalgebras}\label{SS:relcohom} In order to compute extensions 
in ${\mathcal C}$ we use a concrete realization of the relative cohomology for Lie superalgebras. 
Let $\La$ be a Lie superalgebra, $\t\subseteq \La$ be a Lie subsuperalgebra, and $M$ be a $\La$-supermodule in $\mathcal{C}.$ 
For $p\geq 0$ set
$$C^p(\La; M)=\Hom_{\C}(\Lambda_s^p(\La), M),$$
where $\Lambda^{p}_s(\La)$ is the super wedge product.  That is, $\Lambda_s^p(\La)$ is the $p$-fold tensor 
product of $\La$ modulo the $\La$-subsupermodule
generated by elements of the form
$$x_1\otimes \dotsb \otimes x_k\otimes x_{k+1}\otimes \dotsb \otimes x_p+(-1)^{\overline{x_k} \  
\overline{x_{k+1}}}x_1\otimes \dotsb \otimes x_{k+1}\otimes x_{k}\otimes \dotsb \otimes x_p$$
for homogeneous $x_1, \dotsc , x_p\in \La$. Therefore, $x_k, x_{k+1}$ skew commute unless both are odd 
in which case they commute. 

Let $d^p : C^p(\La; M) \rightarrow C^{p+1}(\La; M)$ 
be given by the formula: 

{\setlength\arraycolsep{2pt}\begin{eqnarray}\label{eq:cohom}d^p(\phi )(x_1\wedge  \dotsb  \wedge x_{p+1})&=&\sum_{i<j}(-1)^
{\sigma_{i,j}(x_1,\dotsc ,x_p)}\phi ([x_i, x_j] \ \wedge x_1\wedge  \dotsb  \wedge \hat {x}_i\wedge  \dotsb  \wedge 
\hat {x}_j\wedge \dotsb \wedge x_{p+1}) \nonumber \\ 
& & \quad +\sum _i(-1)^{\gamma (x_1,\dotsc ,x_p,\phi )}x_i\phi (x_1\wedge \dotsb  \wedge \hat {x}_i\wedge \dotsb  \wedge x_{p+1}),
\end{eqnarray}}
where $x_1,\dotsc ,x_{p+1}$ and $\phi$ are assumed to be homogeneous, and
\begin{align*}
\sigma _{i,j}(x_1,\dotsc ,x_p)&:=i+j+\overline{x_i}(\overline{x}_1+\dotsb +\overline{x}_{i-1})+\overline{x_j}
(\overline{x}_1+\dotsb +\overline{x}_{j-1}+\overline{x}_i), \\
\gamma _i(x_1,\dotsc ,x_p,\phi)&:=i+1+\overline{x}_i(\overline{x}_1+\dotsb +\overline{x}_{i-1}+\overline{\phi}).
\end{align*}

Ordinary Lie superalgebra cohomology is then defined as 
$$\HH^p(\La; M)=\Ker \  d^p/\text{Im} \  d^{p-1}.$$
The relative version of the above construction is given as follows. 
Define
$$C^p(\La, \t;M)=\Hom_\t(\Lambda_s^p(\La/\t),M).$$
Then the map $d^p$ induces a well defined map $d^p:C^p(\La,\t;M) \rightarrow C^{p+1}(\La,\t;M)$ and we define
$$\HH^p(\La,\t;M)=\Ker \ d^p/\text{Im} \ d^{p-1}.$$

\subsection{Relating Cohomology Theories} Let $R$ be an associative superalgebra and $S$ a subsuperalgebra. 
Given $R$-supermodules $M$ and $N$ one can define cohomology with respect to the pair $(R,S)$ which we denote 
by $\Ext^{\bullet}_{(R,S)}(M,N)$ (cf.\ \cite[Section 2.2]{BKN1}).  In particular, if ${\mathfrak t}$ is a Lie 
subsuperalgebra of $\fg$ 
then one can define cohomology for the pair $(U(\fg),U({\mathfrak t}))$. The following proposition 
relates the relative cohomology with the cohomology theories of $(U(\fg),U({\mathfrak t}))$ and 
${\mathcal C}_{(\fg,{\mathfrak t})}$. The proof follows from \cite[Proposition 2.4.1]{BKN1}; see also \cite{Kum} 
for the case of ordinary Lie algebras.

\begin{prop}\label{relativecoho} 
Let $\t$ be a Lie subsuperalgebra of $\La$, and $M,N$ be $\La$-supermodules in $\mathcal{C}=\mathcal{C}_{(\fg,\mathfrak{t})}$ 
and assume that $\La$ is finitely semisimple as a $\t$-supermodule under the adjoint action. Then,
\begin{itemize} 
\item[(a)] $\Ext_{(U(\La),U(\t))}^\bullet (M,N)\cong 
 \Ext_{(U(\La),U(\t))}^\bullet (\C,\Hom_{\C}(M, N))\cong \HH^\bullet (\La,\t;\Hom_{\C}(M, N));$
\item[(b)] $\Ext_{\mathcal C}^{\bullet}(M,N)\cong \Ext_{(U(\La),U(\t))}^{\bullet}(M,N).$
\end{itemize} 
\end{prop}


\section{Representation Theory and Atypicality for $W(n)$}\label{S:OnWn}

\subsection{} We begin by recalling the definition of the simple Lie superalgebras of type $W(n)$. As 
a background source we refer the reader to \cite{Kac, Sch, Ser}. 

Assume that $n\geq 2$. The Lie superalgebra $W(n)$ may be described as follows.
Let $\Lambda (n)$ be the exterior algebra of the vector space $V=\C^n$. The algebra 
$\Lambda (n)=\oplus _{k=0}^{n}\Lambda^{k}(n)$ is an associative superalgebra of dimension $2^n$ with a $\Z$-grading given by total degree.
The $\Z_2$-grading is inherited from the $\Z$-grading by setting $\Lambda(n)_{\0} = \oplus_{k} \Lambda^{2k}(n)$ and $\Lambda(n)_{\1} = \oplus_{k} \Lambda^{2k+1}(n)$.    

A (homogeneous) \emph{superderivation} of $\Lambda(n)$ is a linear map $D: \Lambda(n) \to \Lambda(n)$ which satisfies $D(xy)=D(x)y + (-1)^{\overline{D}\; \overline{x}}xD(y)$ for all homogenous $x,y \in \Lambda(n).$  Set $W(n)$ to be the vector space of all superderivations of $\Lambda(n).$  Then $W(n)$ is a Lie superalgebra via the supercommutator bracket. Furthermore, $W(n)$ inherits a $\Z$-grading, 
$$W(n)=W(n)_{-1}\oplus W(n)_0\oplus \dotsb \oplus W(n)_{n-1},$$
from $\Lambda(n)$ by setting $W(n)_k$ to be the superderivations which increase the degree of 
a homogeneous element by $k$.  The $\Z_2$-grading on $W(n)$ is obtained from the $\Z$-grading by setting $W(n)_{\0}= \oplus_{k} W(n)_{2k}$ and $W(n)_{\1}= \oplus_{k} W(n)_{2k+1}.$  One can verify that $[W(n)_{k},W(n)_{l}] \subseteq W(n)_{k+l}$ for all $k,l \in \Z.$  Most importantly this implies $W(n)_{0}$ is a Lie algebra and $W(n)_{k}$ ($k=-1, \dotsc , n-1$) is a $W(n)_{0}$-module under the adjoint action.

Every element of $W(n)$ restricts to a linear map $V \rightarrow \Lambda(n)$. 
Conversely every element of $W(n)$ arises in this way and so one has an isomorphism of vector spaces
$$W(n)\cong \Lambda(n)\otimes V^\ast.$$ 
This identification will be useful for computations. 
Fix an ordered basis $\{\xi _1, \dotsc , \xi _n\}$  for $V$. For each ordered subset $I=\{i_1, \dotsc , i_s\}$ of  
$N=\{1, \dotsc , n\}$ with $i_1<i_2< \dotsb  <i_s$,
let $\xi _I=\xi _{i_1}\xi _{i_2}\dotsb \xi _{i_s}$. The set of all such $\xi_I$ forms a basis for $\Lambda(n)$. 
For $1\leq i\leq n$ let $\partial _i$ be the
element of $W(n)$ such that $\partial _i(\xi_j)=\delta _{ij}$. An explicit basis for $\Lambda(n)\otimes  V^\ast$ 
is then given by the set of all
$\xi_I\otimes \partial _i$, where here we identify $\partial _i$ with its restriction to $V$. 
We shall write $\xi_I\partial _i$ instead of $\xi_I\otimes \partial _i$. We use the isomorphism above to 
identify $W(n)$ and $\Lambda(n)\otimes  V^\ast$.

In particular, one has $W(n)_0\cong V\otimes V^\ast\cong \gl$ and the element $\xi_i\partial _j$ 
corresponds to the matrix unit $e_{i,j}$ (i.e.\ the matrix with a one in the $(i,j)$ position and zeros elsewhere). Also $W(n)_{-1}\cong V^\ast$ as a $W(n)_0$-module. In general the basis 
elements $\xi_I\partial _i$ belonging to $W(n)_{k}$
are those with $|I| =k+1.$  Thus $\dim_{\C} W(n)_k=n\binom{n}{k+1}$ and $\dim_{\C} W(n)=n2^n$.

With the exception of Section~\ref{S:realization} we use the following notational conventions throughout. 
Set $\fg=W(n)$ with $\fg_{i}=W(n)_{i}$, 
$i\in {\mathbb Z}$, and $\fg_{\bar{i}}=W(n)_{\bar{i}}$, $\bar{i}\in {\mathbb Z}_{2}$. Moreover, 
let $\La^{+}=\La_1\oplus \dotsb \oplus \La_{n-1}$ and  $\La^{-}=\La_{-1}$, so that $\fg$ has 
the lopsided triangular decomposition 
\[
\fg=\fg^{-}\oplus \fg_{0} \oplus \fg^{+}.
\] Throughout the paper all $\La$-supermodules will be assumed to be objects in the category $\mathcal{C}=\mathcal{C}_{(\La,\La_0)}$. 

\subsection{Kac supermodules and Finite Dimensional Simple $\La$-supermodules} 
Fix the maximal torus $\fh \subseteq \fg_{0}$ consisting of diagional matrices and the Borel subalgebra $\b_0$ of $\La_0$ consisting of upper triangular matrices. Let 
$X_0^+ \subset \fh^{*}$ denote the parametrizing set of highest weights for the simple finite dimensional $\La_0$-modules with respect to the pair $(\fh , \b_0)$ and let $L_0(\la)$ denote the simple finite dimensional $\fg_{0}$-module with 
highest weight $\la\in X_{0}^{+}.$ We view $L_{0}(\lambda)$ as a $\fg_{0}$-supermodule concentrated in degree $\0$.  

The \emph{Kac supermodule} $K(\la )$ is the induced representation of $\fg$,
$$K(\lambda )=U(\La)\otimes _{U(\fg_{0}\oplus \fg^{+})}L_0(\la ),$$ 
where $L_0(\la )$ is viewed as a $\fg_{0}\oplus \fg^{+}$ via inflation through the canonical quotient map $\fg_{0}\oplus \fg^{+} \to \fg_{0}$. 
By the PBW theorem for Lie superalgebras the supermodule $K(\la )$ is a finite dimensional indecomposable object in ${\mathcal C}_{(\fg,\fg_{0})}.$  With respect to the choice of Borel subalgebra $\b_0\oplus \fg^{+} \subseteq \fg$ one has a dominance order on weights.  With respect to this ordering $K(\lambda)$ has highest weight $\la$ and, therefore, a unique simple quotient which we denote by $L(\la )$.
Conversely, every finite dimensional simple supermodule appears as the head of some Kac supermodule (cf.\ \cite[Theorem 3.1]{Ser}). 

From our discussion above one observes that the set
$$\{L(\la)\mid \la \in X_0^+\}$$ is a complete irredundant collection of simple finite dimensional $\La$-supermodules.

\subsection{Root Decomposition}

Recall from the previous section that we fixed a maximal torus $\fh \subseteq \fg_{0} \subseteq \fg.$  With respect to this choice we have a root decomposition
$$\La=\Lsa\oplus \bigoplus_{\alpha \in\Phi }\La_{\alpha}.$$ 
Many properties of root decompositions for semisimple Lie algebras do not hold in our case. For example, a root can have multiplicity
bigger than one, and $\a \in \Phi$ does not imply that $-\a\in \Phi$. Still any root space $\La_\a$
is concentrated in either degree $\0$ or degree $\1$ and in this way one can define a natural parity function on roots.

Let us describe the roots. We choose the standard basis $\varepsilon _1,\dotsc , \varepsilon _n$ of $\Lsa^\ast$ where $\varepsilon_{i}(\xi_{j}\partial_{j})=\delta_{i,j}$ for all $1 \leq i,j \leq n.$  Then the root system of $\fg$ is the set 
$$\Phi=\{\varepsilon _{i_1}+ \dotsb  +\varepsilon _{i_k}-\varepsilon _j\mid 1\leq i_1<\dotsb <i_k\leq n,\ 1\leq j\leq n\}.$$
The  set of simple roots for $\La$ is
$$\triangle =\{\varepsilon _1-\varepsilon _2, \dotsc , \varepsilon _{n-1}-\varepsilon _n\}.$$

\subsection{Typical and Atypical Weights} We consider Borel subalgebras $\b$ of $\La$ containing $\b_0$. Among such subalgebras 
we distinguish $\b_{\max}=\b_0\oplus \fg^{+}$ and
$\b_{\min}=\b_0\oplus \fg^{-}$. Let $\b$ denote either $\b_{\max}$ or $\b_{\min}$. Then $\lambda \in \fh^{*}$ defines a 
one dimensional representation of $\b$ which we denote by $\C_ {\lambda}$. The induced supermodule $M^{\b}(\la)=U(\La)\otimes _{U(\b)}\C_{\la}$ 
has a unique proper maximal submodule. We denote the unique irreducible quotient by $L^{\b}(\la)$. In particular, if $\lambda \in X_{0}^{+}$, then  $L(\la)\cong L^{\b_{\max}}(\la).$ 

Denote by $\la'$ the weight such that $L^{\b_{\min}}(\la')\cong L(\la)$.
Let $\Phi(\La_{-1})$ be the set of roots which lie in $\La_{-1}$. By Serganova \cite[Lemma 5.1]{Ser}, 
\begin{equation}\label{eq:typ}\la'=\la+\sum_{\alpha \in \Phi(\La_{-1})}\a
\end{equation}
for a Zariski open set of $\la \in \Lsa^\ast$. Here $\Phi(\fg_{-1})$ denotes the set of weights of $\fg_{-1}.$  Following Serganova, we call 
$\la \in \Lsa^\ast$ \emph{typical} if \eqref{eq:typ} holds for $\la$ and otherwise 
$\la$ is \emph{atypical}.  Serganova determines a necessary and sufficient combinatorial condition for $\lambda$ to be typical. Namely, by \cite[Lemma 5.3]{Ser} one has that the set of atypical weights for $\La$ is
$$\Omega =\{a\varepsilon_i+\varepsilon_{i+1}+ \dotsb  +\varepsilon_n \in \fh^{*}\mid a\in \C,\ 1\leq i\leq n\}.$$  In particular, one has
$$\Omega\cap X_{0}^{+}=\{a\varepsilon_{1}+\dots +\varepsilon _n \mid a = 1, 2, 3,\dotsc \}\cup 
\{b\varepsilon _n \mid\ b=0,-1,\dots \}.$$
\section{Cohomology in ${\mathcal C}_{(\La,\La_0)}$} \label{S:CohominC}

\subsection{} The goal of this section is to compute the cohomology ring $R=\HH^{\bullet}(\fg,\fg_0;\C ).$  The main result is Theorem~\ref{cohomologyring} which shows that $R$ can be identified with a ring of invariants.  Consequently one sees that $R$ is finitely generated and $\HH^{\bullet}(\fg ,\fg_{0};M)$ is a finitely generated $R$-module for any finite dimensional $\fg$-supermodule $M.$

\subsection{} We begin by showing that the calculation of $\fg_{0}$-invariants on  
$\Lambda_{s}^{\bullet}((\fg/\fg_{0})^{*})$ reduces to looking at $\fg_{0}$-invariants on $\Lambda^{\bullet}_{s}\left( \fg_{-1}^{*}\oplus \fg^{*}_{1}\right)$. This will be accomplished by 
using information from the representation theory of $\fg_{0}\cong \gl$. 

\begin{thm}\label{cutinvariants} Let $\fg =W(n)$ and let $p \geq 0.$ Then, $\Lambda_{s}^{p}((\La/\La_0)^\ast)^{\La_0}\cong \Lambda_{s}^{p}(\La_{-1}^\ast \oplus \La_1^\ast )^{\La_0}$
\end{thm}
\begin{proof} First observe that $\La/\La_0\cong \La_{-1} \oplus \La_1 \oplus \La_2 \oplus \dotsb  \oplus \La_{n-1}$ as
$\La_0$-modules.  We then have

\begin{align*}\Lambda _s^{p}(\La_{-1}^\ast \oplus \La_1^\ast \oplus & \La_2^\ast \oplus \dotsb 
\oplus \La_{n-1}^\ast )^{\La_0} \\
& \cong  \bigoplus\left(\Lambda _s^{i_{-1}}(\La_{-1}^\ast) \otimes 
\Lambda _s^{i_1}(\La_1^\ast)\otimes\Lambda_s^{i_2} (\La_2^\ast)\otimes \dotsb  \otimes \Lambda _s^{i_{n-1}}(\La_{n-1}^\ast) \right)^{\La_0} \\
& \cong  \bigoplus  \Hom_{\La_0} \left( \C,
\Lambda _s^{i_{-1}}(\La_{-1}^\ast) \otimes 
\Lambda _s^{i_1}(\La_1^\ast)\otimes\Lambda_s^{i_2} (\La_2^\ast)\otimes \dotsb  \otimes \Lambda _s^{i_{n-1}}(\La_{n-1}^\ast)\right) \\
& \cong  \bigoplus  \Hom_{\La_0}\left(\Lambda _s^{i_{-1}}(\La_{-1}),\Lambda _s^{i_1}(\La_1^\ast)\otimes \Lambda_s^{i_2} (\La_2^\ast)\otimes \dotsb  \otimes 
\Lambda _s^{i_{n-1}}(\La_{n-1}^\ast) \right)
\end{align*}
where the direct sums are taken over all nonnegative integers $i_{-1}, i_1, \dotsc , i_{n-1}$ such that $i_{-1}+i_1+\dotsb +i_{n-1}=p$.

Recall that $\La_{-1}$ is isomorphic to the dual of the natural $\fg_0$-module and, since $\fg_{-1}$ is concentrated in degree $\1 ,$  one has $\Lambda _s^{i_{-1}}(\La_{-1})\cong S^{i_{-1}}(\La_{-1})$ as $\fg_{0}$-modules.  It is 
well known that symmetric powers of the dual of the natural module are simple (cf. \cite[II 2.16]{Jan}) 
and so $\Lambda _s^{i_{-1}}(\La_{-1})$ is a simple $\La_0$-module with highest weight 
\[
\mu =(\mu_1, \dotsc , \mu_n)=(0, \dotsc , 0,-i_{-1}).
\]

Since $\La_k^\ast\cong \Lambda ^{k+1}(V^\ast)\otimes V \ \  (-1\leq k\leq n-1)$, where $V$ is the 
natural $\La_0$-module, and since $\Lambda^{k+1}(V^\ast)$ (resp. $V$) are simple 
$\La_0$-modules with highest weights $(0,\dotsc , 0,-1, \dotsc , -1)$ (resp. $(1,0,\dotsc , 0)$), the highest weight 
of $\La_k^\ast$ will be the sum of these two weights,
$(1,0,\dotsc ,0, -1, \dotsc , -1)$. Therefore the largest possible weight that could occur in the weight space decomposition of the $\La_0$-module
\[
A:=(\La_1^\ast)^{\otimes i_{-1}}\otimes (\La_2^\ast)^{\otimes i_{2}}\otimes \dotsb  \otimes (\La_{n-1}^\ast)^{\otimes i_{n-1}}
\]
is \begin{eqnarray*}
\lambda & = & (\lambda_1, \dotsc , \lambda_n)\\
& = & (i_1, 0, \dotsc ,0, -i_1,-i_1)+(i_2,0, \dotsc , 0, -i_2, -i_2, -i_2)+\dotsb +(0,-i_{n-1},\dotsc ., -i_{n-1})\\
& = &(\Sigma _{t=-1}^{n-1}i_t, -i_{n-1}, -i_{n-2}-i_{n-1}, \dotsc ,-\Sigma _{t=-2}^{n-1}i_t, -\Sigma _{t=-1}^{n-1}i_t, -\Sigma _{t=-1}^{n-1}i_t).\end{eqnarray*}

If the $\Hom$-space 
\begin{equation}\label{eq:head}
\Hom_{\La_0}(\Lambda _s^{i_{-1}}(\La_{-1}),\Lambda _s^{i_1}(\La_1^\ast)\otimes\Lambda _s^{i_2}(\La_2^\ast)\otimes \dotsb  \otimes 
\Lambda _s^{i_{n-1}}(\La_{n-1}^\ast))
\end{equation}
 is non-zero, then the weight $\mu$ occurs in the weight space decomposition of the $\La_0$-module $A$.
However, if $\mu$ occurs in the weight space decomposition of this $\La_0$-module, then it has to be less than or equal to $\lambda$ in the dominance order;  that is, $\Sigma _{i=1}^k\mu_i\leq \Sigma _{i=1}^k\lambda_i$ for all $k>0$. By considering this inequality when $k=n-1$ one obtains 
$$0\leq -i_2-2i_3-\dotsb -(n-2)i_{n-1}.$$Thus the  $\Hom$ space in \eqref{eq:head} is nonzero only when
$i_2=i_3=\dotsb =i_{n-1}=0$. This gives the stated result.
\end{proof}

\subsection{Calculation of $\HH^{\bullet}(\fg,\fg_{0};\C )$} The previous theorem can be used to show that the cohomology ring $\HH^\bullet (\La,\La_0;\C)$  
can be identified with a ring of invariants.  Recall that $G_{0} \cong \text{GL}(n)$ denotes the connected reductive group such that $\Lie (G_{0}) =\fg_{0}$ and the adjoint action of $G_{0}$ on $\fg$ differentiates to the adjoint action of $\fg_{0}$ on $\fg.$
 
\begin{thm}\label{cohomologyring} Let $\fg =W(n).$ Then, 
\[
\HH^\bullet (\La,\La_0;\C)\cong S((\La_{-1} \oplus \La_1)^\ast )^{\La_0}=S((\La_{-1} \oplus \La_1)^\ast )^{G_0}.
\]
\end{thm}
\begin{proof}  By Theorem ~\ref{cutinvariants} one has
\begin{align*}C^p(\La, \La_0; \C)& =  \Hom_{\La_0}(\Lambda_s^p(\La /\La_0), \C) \\
& \cong \Lambda_{s}^{p} \left((\fg / \fg_{0})^{*} \right)^{\fg_{0}}\\
& \cong\Lambda_{s}^{p} \left(\fg_{-1}^{*} \oplus \fg_{1}^{*} \right)^{\fg_{0}}\\
& \cong \Hom_{\La_0}(\Lambda_s^p(\La_{-1} \oplus \La_1) ,\C). 
\end{align*}
Now observe that in this case the differential $d^p$ in \eqref{eq:cohom} is identically zero. Namely, in the first sum of \eqref{eq:cohom} each $[x_i, x_j]$ is zero in the quotient $\La/ \La_0$ since the bracket preserves the $\Z$-grading and the terms in the second sum of \eqref{eq:cohom} are zero since here $M$ is the trivial supermodule. 

As a consequence the cohomology can be identified with the cochains. It remains to observe that since $\fg_{-1}\oplus\fg_{1}$ is concentrated in degree $\1,$ one has 
\[
C^{p}(\fg,\fg_0; \C ) \cong \Lambda_s^p\left( (\La_{-1} \oplus \La_1)^{*}\right)^{\fg_{0}} \cong S^{p}\left( (\La_{-1} \oplus \La_1)^{*}\right)^{\fg_{0}}=S^{p}\left( (\La_{-1} \oplus \La_1)^{*}\right)^{G_{0}}.
\]  
\end{proof}

\subsection{Finite Generation Results} Let $M$ be a finite dimensional $\La$-supermodule. By using the Yoneda product $\HH^\bullet(\La, \La_0; M)$ is a module for the cohomology ring $R.$  A result from invariant theory shows that this module is finitely generated over $R$. 

\begin{thm}\label{finitegeneration}
Let $M$ be a finite dimensional $\La$-supermodule. Then,
\begin{itemize} 
\item[(a)] The superalgebra $\HH^\bullet (\La,\La_0;\C)$ is a finitely generated commutative ring;
\item[(b)] The cohomology $\HH^\bullet (\La,\La_0;M)$ is finitely generated as an $\HH^\bullet (\La,\La_0;\C)$-module.
\end{itemize} 
\end{thm}

\begin{proof} Since $G_{0}\cong \text{GL}(n)$ is reductive, (a) follows from Theorem~\ref{cohomologyring} and the classical invariant theory 
result of Hilbert \cite[Theorem 3.6]{PV}. 

(b) Observe that 
\begin{align*}
\Hom_{\C}(\Lambda_s^{\bullet}(\fg/\fg_{0}),M)&\cong \Lambda_s^{\bullet} ((\fg/\fg_{0})^{*})\otimes M \\
       & \cong \Lambda_s^{\bullet} (\fg_{\1}^{*}) \otimes \Lambda_s^{\bullet} ((\fg_{\0}/\fg_{0})^{*})\otimes M \\
        &  = S (\fg_{\1}^{*}) \otimes \Lambda_s^{\bullet} ((\fg_{\0}/\fg_{0})^{*})\otimes M 
\end{align*}
is finitely generated as a $S(\fg_{\1}^{*})$-module (under left multiplication) since $\Lambda_s^{\bullet} ((\fg_{\0}/\fg_{0})^{*})\otimes M$ is finite dimensional. Since $S(\fg_{\1}^{*})$ is a finitely generated commutative $G_{0}$-algebra, one can invoke \cite[Theorem 3.25]{PV} to see that 
\[
\Hom_\C(\Lambda_s^{\bullet}(\fg/\fg_{0}),M)^{G_{0}}=\Hom_{G_0}(\Lambda_s^{\bullet}(\fg/\fg_{0}),M) = C^{\bullet}(\fg ,\fg_{0};M)
\] is finitely generated as a $S(\fg_{\1}^{*})^{G_{0}}=\Lambda_{s}^{\bullet}(\fg_{\1}^{*})^{G_{0}}\cong \HH^\bullet(\La,\La_0;\C)$-module. One can now argue as in the proof of \cite[Theorem 2.5.3]{BKN1} to infer that $\HH^{\bullet}(\fg ,\fg_{0};M)$ is a finitely generated $ \HH^\bullet(\La,\La_0;\C)$-module.
\end{proof}

\section{Invariant Theory Calculations}\label{S:invarianttheory}

\subsection{} Recall from Theorem~\ref{cohomologyring} that
\begin{equation}\label{E:cohomiso}
R=\HH^{\bullet}(\La ,\La_{0};\C)\cong S\left((\fg_{-1}\oplus \fg_{1})^{*} \right)^{\La_{0}}.
\end{equation} Thus to compute $R$ it suffices to compute the invariant ring on 
the right hand side of \eqref{E:cohomiso}.  To do so we use a result of Luna and Richardson \cite{LR}.  First we require certain preliminaries.


\subsection{} If $G$ is an algebraic group which acts on a variety $X,$ then we write $g.x$ for the action of $g \in G$ on the element $x \in X.$  Set 
\[
\Stab_{G}(x) = \{g \in G \mid g.x=x \},
\] the \emph{stabilizer} of $x.$  An element $x \in X$ is \emph{semisimple} if the orbit $G.x$ is closed in $X$.  An element $x \in X$ is said to be \emph{regular} if the dimension of the orbit $G.x$ is of maximal possible dimension among all orbits.  Equivalently, $x$ is regular if the dimension of $\Stab_{G}(x)$ is of minimal dimension among all stabilizer subgroups.

The group $G_{0}\cong GL(n)$ acts on $\La$ by the adjoint 
action and its action preserves the $\Z$-grading of $\La.$  
Let 
\[
\beta: \La_{-1} \oplus \La_{1} \to \La_{0}
\]
be the $G_{0}$-equivariant map given by $\beta(x+y)=[x,y]$ for all $x \in \La_{-1}$ and $y \in \La_{1}.$  

Fix $T \subseteq G_{0}$ to be the maximal torus consisting of all diagonal matrices.  Then $\Lsa = \operatorname{Lie}\left(T \right),$ the Cartan subalgebra we fixed in Section~\ref{S:OnWn}.  

The following lemma summarizes some well known results about the adjoint action of $G_{0}$ on $\mathfrak{g}_{0}.$  See, for example, \cite{CM,Hum}.

\begin{lem}\label{L:basiclemma} Let $h \in \fh.$ Then,
\begin{itemize}
\item[(a)]  The element $h$ is regular if and only if $\Stab_{G_{0}}(h)=T;$
\item[(b)] The element $h$ is regular if and only if all the eigenvalues of $h$ are pairwise distinct elements of $\C;$
\item[(c)] An element $x \in \La_{0}$ is semisimple if and only if it is $G_{0}$-conjugate to an element of $\Lsa;$
\item[(d)] An element of $x \in \La_{0}$ is semisimple and regular if and only if it is $G_{0}$-conjugate to a regular element of $\Lsa;$
\item [(e)]  The semisimple regular elements of $\fg_{0}$ form a dense open set in $\fg_{0}$.
\end{itemize}
\end{lem}We leave it to the reader to verify the following basic observations.

\begin{lem}\label{L:basiclemma2}  Let $G$ be an algebraic group acting on the varieties $X$ and $Y.$  Let $f: X \to Y$ be a $G$-equivariant map.  
Then the following statements hold true.
\begin{itemize}
\item [(a)] If $y \in Y$ and $x \in  f^{-1}(y),$ then 
\[ \Stab_{G}(x) \subseteq \Stab_{G}(y).
\]
\item [(b)] If $y \in Y,$ then 
\[
G.f^{-1}(y)=f^{-1}\left(G.y \right).
\]  In particular, if $x \in f^{-1}(y),$ then $G.x \subseteq f^{-1}(G.y).$
\end{itemize}
\end{lem}


\subsection{}\label{SS:semisimple}

We saw in Section~\ref{S:OnWn} that $\La_{-1}\cong V^{*}$ as a $\mathfrak{g}_{0}$-module and has 
basis $\partial_{i}$, $\La_{0}$ has basis $\xi_{i}\partial_{j},$ and $\La_{1}=\Lambda^{2}(V) \otimes V^{*}$ with basis $\xi_{i}\xi_{j}\partial_{k}$ (with $i < j$), where $1 \leq i,j,k\leq  n.$  Recall that the 
isomorphism $\La_{0} \cong \gl$ is given by $\xi_{i}\partial_{j} \leftrightarrow e_{i,j},$ where $e_{i,j}$ is the $(i,j)$ 
matrix unit.  In particular, $\Lsa$ is spanned by the set $\{ \xi_{i}\partial_{i} \mid 1 \leq i \leq n \}.$

\begin{lem}\label{L:semisimplereg}  Let $h \in \Lsa$ be a semisimple regular element and write 
\[
h = \sum_{i=1}^{n}c_{i}\xi_{i}\partial_{i},
\] with $c_{i}\in \C .$ One then has the following.
\begin{itemize}
\item[(a)] If $c_{1}, \dotsc , c_{n}$ are all nonzero, then $\beta^{-1}(h) = \emptyset.$
\item[(b)] If $c_{1}=0$ and $x \in \beta^{-1}(h),$ then 
\begin{equation}\label{E:betainverse}
x = a_{1}\partial_{1} + \sum_{l=2}^{n} \frac{c_{l}}{a_{1}}\xi_{1}\xi_{l}\partial_{l} + \sum_{\substack{r,s,t\\  1 < r < s}}b_{r,s,t}\xi_{r}\xi_{s}\partial_{t},
\end{equation} where $a_{1}, b_{r,s,t} \in \C$ and $a_{1}\neq 0.$ 

\end{itemize}
\end{lem}
\begin{proof}  We only sketch the calculation here.  First, let $x \in \beta^{-1}(h)$  
and write $x \in \fg_{-1}\oplus \fg_{1}$ in our preferred basis:
\begin{equation}\label{E:xinhfiber}
x = \sum_{i} a_{i}\partial_{i} + \sum_{\substack{i,j,k \\ i < j}}b_{i,j,k}\xi_{i}\xi_{j}\partial_{k}.
\end{equation}
By a direct calculation of $\beta(x),$ one sees that if $a_{t} \neq 0$ for some $1 \leq t \leq n,$ then necessarily $c_{t}=0.$  
Therefore, if all the coeffients of $h$ are nonzero, then there is no $x$ such that $\beta(x)=h.$  This proves part (a).  Now say $c_{1}=0.$  Since $h$ is regular, $c_{2},\dotsc ,c_{n}$ are all nonzero by Lemma~\ref{L:basiclemma}(b). 
But then by the proof of part (a) one has $a_{2}=\dotsb = a_{n}=0.$  This observation simplifies the calculation of $\beta(x).$  
Doing so and using that the image is equal to $h,$ one obtains \eqref{E:betainverse}.
\end{proof}

\begin{prop}\label{P:semisimple}  Let $h \in \Lsa$ be a semisimple regular element as in part (b) of the previous lemma.  
Let $x_{0} \in \beta^{-1}(h)$ be chosen so that all the coefficents $b_{r,s,t}$ are zero in \eqref{E:betainverse}.  Then $x_{0} \in \fg_{-1}\oplus \fg_{1}$ is a semisimple element.
\end{prop}

\begin{proof} First one computes $\Stab_{G_{0}}(x)$ for any $x \in \beta^{-1}(h).$  By
 Lemma~\ref{L:basiclemma2} and Lemma~\ref{L:basiclemma}(a)  one has $\Stab_{G_{0}}(x) \subseteq \Stab_{G_{0}}(h)= T.$  By part (b) of the previous lemma $x$ is a linear combination of distinct weight vectors. From this one sees that $t=\operatorname{diag}(t_{1}, \dotsc , t_{n}) \in T$ fixes $x$ if $t_{1}=1.$  
This is also sufficient in the case of $x_{0}$.  Otherwise there will be additional contraints on $t$ 
and the stabilizer will be a proper, smaller dimensional subgroup of 
\begin{equation}\label{E:Tnminus1}
T_{n-1}:=\{t=\operatorname{diag}(t_{1}, \dotsc , t_{n})  \in T \mid t_{1}=1 \}.
\end{equation}
Thus $x_{0}$ has maximal stabilizer dimension and, hence, minimal orbit dimension in the closed set $\beta^{-1}(G.h).$   
It follows that $G.x_{0}$ must be closed.
\end{proof}

\subsection{}\label{SS:principal} Let $G$ be a reductive algebraic group acting on an affine variety 
$X$.  Let $\pi: X \to X/G$ be the canonical quotient map.  An element $\zeta \in X/G$ is said to be \emph{principal} if 
there is an open neighborhood $U$ such that $\zeta \in U \subseteq X/G$ and 
for any semisimple $x , y \in \pi^{-1}(U),$ the groups $\Stab_{G}(x)$ and $\Stab_{G}(y)$ are conjugate in $G$ \cite[Definition 3.2, Remark 3.3]{LR}.  
 Let $(X/G)_{\text{pr}}$ denote the set of principal elements of $X/G.$  By \cite[Lemma 3.4]{LR} $(X/G)_{\text{pr}}$ is a 
nonempty, dense, open subset of $X/G.$

Let 
\begin{align*}
\pi : \fg_{-1} \oplus \fg_{1} \to (\fg_{-1} \oplus \fg_{1})/G_{0} && \text{and} && p: \La_{0} \to \La_{0}/G_{0}
\end{align*}
denote the canonical quotient morphisms.  Let $\varphi: (\fg_{-1} \oplus \fg_{1})/G_{0} \to \La_{0}/G_{0}$ be the morphism induced by the map $p \circ \beta: \fg_{-1} \oplus \fg_{1} \to \La_{0}/G_{0}.$  That is, the following diagram commutes.
\begin{figure}[ht]
\setlength{\unitlength}{.5cm}
\begin{center}
\begin{picture}(17,5)
\put (3,4){$\fg_{-1} \oplus \fg_{1}$}
\put (10,4){$\fg_{0}$}
\put (16,4){$\fg_{0}/G_{0}$}
\put (2.5,0){$(\fg_{-1} \oplus \fg_{1})/G_{0}$}
\put(6.5,4){\vector(1,0){2.75}}
\put(12,4){\vector(1,0){2.75}} 
\put(3.5,3.4){\vector(0,-1){2.5}}
\put(5,1){\vector(4,1){10}}

\put (7.5,4.5){$\beta$}
\put (13,4.5){$p$}
\put (2.5,2){$\pi$}
\put (10,1.5){$\varphi$}
\end{picture}
\end{center}
\end{figure}

We observe that it follows from Lemma~\ref{L:basiclemma} that the set $(\fg_{0}/G_{0})_{\text{pr}}$ is precisely the image under $p$ of the semisimple regular elements of $\Lsa$.  

\begin{prop}\label{P:principal}  Let $x_{0} \in \fg_{-1} \oplus \fg_{1}$ be as in Proposition~\ref{P:semisimple}.  Then $\pi(x_{0})$ is a principal element of $(\fg_{-1} \oplus \fg_{1})/G_{0}.$
\end{prop}

\begin{proof}   
Let $U:=\varphi^{-1}\left((\fg_{0}/G_{0})_{\text{pr}} \right)$.  By definition, $\beta(x_{0})$ is semisimple and 
regular, so $p(\beta(x_{0}))=\varphi(\pi(x_{0}))$ is principal in $\La_{0}/G_{0}$.  That is, $\pi(x_{0}) \in U.$  
Therefore $U$ is a nonempty open neighborhood of $\pi(x_{0})$ in $(\fg_{-1} \oplus \fg_{1})/G_{0}.$  Let $\zeta \in  U$ and let $y \in \fg_{-1} \oplus \fg_{1}$ be 
 a semisimple element in $\pi^{-1}(\zeta).$  Then $\varphi(\pi(y))=p(\beta(y)),$ so $\beta(y) \in p^{-1}(\eta)$ for some 
principal $\eta \in \La_{0}/G_{0}.$  But then $\eta = p(h)$ for some semisimple regular $h \in \Lsa.$  However since $h$ is semisimple 
and regular, it follows that $p^{-1}(\eta)=G_{0}.h.$   That is, up to $G_{0}$-conjugacy one can assume $\beta(y)$ is a semisimple 
regular element of $\Lsa$.  However this implies that $y$ is of the form given in Lemma~\ref{L:semisimplereg} and the stabilizer of such elements was computed in the proof of Lemma~\ref{P:semisimple}.  By that calculation and the fact that $y$ is semisimple one sees that the stabilizer of $y$ is $T_{n-1}.$  Therefore all semisimple elements in the fibers of $U$ have stabilizer conjugate to $T_{n-1}$ and so $x_{0}$ is principal.
\end{proof}

\subsection{}\label{SS:precalculatingR}  The stage is now set to apply the results of Luna and Richardson \cite[Corollary 4.4]{LR} to calculate $R.$ To do so requires certain preliminary calculations. Let $x_{0} \in \fg_{-1} \oplus \fg_{1}$ be the semisimple element fixed in the previous section. 

Let 
\begin{equation}\label{E:Hdef}
H= \Stab_{G_{0}}(x_{0})=T_{n-1},
\end{equation} where the last equality is by the calculations made in the proof of Proposition~\ref{P:semisimple}.
Let 
\begin{equation}\label{E:Ndef}
N=\operatorname{Norm}_{G_{0}}(H) = \left\{g \in G_{0} \mid gHg^{-1}=H \right\}.
\end{equation}  Let us first calculate the group $N.$  

\begin{lem}\label{L:calculatingN}  Let $N=\operatorname{Norm}_{G_{0}}(H).$  Recall that $T$ is the torus of $G_{0}.$  Let $\Sigma_{n}$ be 
the permutation matrices of $G_{0}$ and let $\Sigma_{n-1}$ be the permutation matrices which normalize $T_{n-1}$. 
Then,
\[
N=T\Sigma_{n-1}.
\]
\end{lem}

\begin{proof}  The first step is to prove that $N \subseteq \Norm_{G_{0}}(T)$.  Fix a semisimple 
regular element $t_{0} \in T_{n-1}$ (for the action of $G_{0}$ on itself by conjugation).  
Then $T=\Stab_{G_{0}}(t_{0})$.  Let $n \in N$.  We claim that $nTn^{-1}$ fixes $t_{0},$ 
hence $nTn^{-1}=T$, hence $n \in \Norm_{G_{0}}(T)$.  Let $t \in T$ and consider
\[
(ntn^{-1})t_{0}(ntn^{-1})^{-1}=ntn^{-1}t_{0}nt^{-1}n^{-1}.
\]  However, since $t_{0} \in T_{n-1}$ and $n^{-1} \in \Norm_{G_{0}}(T_{n-1}),$ one has 
that $n^{-1}t_{0}n \in T_{n-1} \subseteq T;$  since $t \in T$ and $T$ fixes $T$ 
pointwise under conjugation, one has $tn^{-1}t_{0}nt^{-1}=n^{-1}t_{0}n.$  Thus, 
\[
ntn^{-1}t_{0}nt^{-1}n^{-1}=nn^{-1}t_{0}nn^{-1}=t_{0}.
\]  Therefore, $ntn^{-1} \in \Stab_{G_{0}}(t_{0})=T.$  That is, as discussed above, $n \in \Norm_{G_{0}}(T)=T\Sigma_{n}.$

One can now verify that $T$ fixes $H$ pointwise, and that the elements of $\Sigma_{n}$ which stabilize $H$ 
are precisely $\Sigma_{n-1}$.
\end{proof}
We next need to calculate $\ff_{\1}:= (\fg_{-1} \oplus \fg_{1})^{H}.$
\begin{lem}\label{L:calculatingF}  The subvariety $\Lta_{\bar{1}}=(\fg_{-1} \oplus \fg_{1})^{H}$ is the $\C$-span of the vectors 
$$\{\partial _1, \xi_1\xi_i\partial _i\mid i=2, \dotsc , n\}.$$

\end{lem}

\begin{proof}  Since $H=T_{n-1},$ $\Lta_{\bar{1}}$ is simply the span of all weight zero vectors with respect to this torus.  
Using the fixed basis of weight vectors for $\fg_{-1} \oplus \fg_{1}$ established in Section~\ref{S:prelims} it 
can be seen that $\Lta_{\bar{1}}$ is spanned by the given vectors.
\end{proof}

\subsection{Explicit Description of $\HH^{\bullet}(\fg,\fg_{0};\C )$} We can now give an explicit description of the cohomology ring $R$. Let $Y_{i} \in \ff_{\1}^{*}$ be given by $Y_{i}(\xi _1\xi _j\partial _j) = \delta_{i,j}$ ($i,j = 2, \dotsc , n$) and $Y_{i}(\partial_{1})=0.$  Let $\partial_{1}^{*} \in \ff_{\1}^{*}$ be given by $\partial_{1}^{*}(\xi _1\xi _j\partial _j)=0$ for all $j=2, \dotsc ,n$ and $\partial_{1}^{*}(\partial_{1})=1.$

\begin{thm}\label{T:calculatingR}  Restriction of functions defines an isomorphism,
\begin{equation*}
\HH^{\bullet}(\fg ,\fg_{0};\C ) \cong S(\Lta_{\bar{1}}^{*})^{N} = \C[Y_{2}\partial_{1}^{*}, \dotsc , Y_{n}\partial_{1}^{*}]^{\Sigma_{n-1}},
\end{equation*}
where $\Sigma_{n-1}$ acts on $Y_{2}\partial_{1}^{*}, \dotsc , Y_{n}\partial_{1}^{*}$ by permutations.  
In particular, $R$ is a polynomial ring in $n-1$ variables of degree $2, 4, \dotsc , 2n-2.$
\end{thm}

\begin{proof}  The first isomorphism follows from \eqref{E:cohomiso} and \cite[Corollary 4.4]{LR}.  Namely, $x_{0} \in \fg_{-1} \oplus \fg_{1}$ is a semisimple element with $\pi(x_{0})$ a principal element of $(\fg_{-1} \oplus \fg_{1})/G_{0}$ so it follows by \cite[Corollary 4.4]{LR} that restriction of functions defines an isomorphism between $S((\fg_{-1} \oplus \fg_{1})^{*})^{G_{0}}$ and $S(\Lta_{\bar{1}}^{*})^{N}.$
Since $T$ is a normal subgroup of $N,$ one can first compute that $S^{\bullet}(\Lta_{\bar{1}}^{*})^{T}= \C[Y_{2}\partial_{1}^{*}, \dotsc , Y_{n}\partial_{1}^{*}]$  and check that $\Sigma_{n-1}$ acts on this ring by permuting the variables.
\end{proof}

\subsection{Detecting Subalgebra for $W(n)$} In \cite[Section 8]{BKN1} it was shown that the simple 
classical Lie superalgebras have either a polar or stable action of $G_{\0}$ on $\fg_{\bar{1}}$. 
As a consequence of this phenomenon one can show that there exists Lie subsuperalgebras 
such that the restriction homomorphism identifies the cohomology ring of $\fg$ with the invariants of the cohomology ring of the detecting subalgebra under the action of some finite pseudoreflection group.  In the case of $\fg = W(n),$ the action of $G_{0}$ on $\fg /\fg_{0}$ is neither polar nor stable.  Nevertheless, one can construct a similar detecting subalgebra for $\La$. 

Let $\ff_{\1}=(\fg_{-1} \oplus \fg_{1})^{H} \subset \fg_{\1}$ be the subspace calculated in Lemma~\ref{L:calculatingF} and let 
\[
\ff_{\0}=\Lie(N) =\Lie (T)= \fh \subset \fg_{\0}.
\] One can verify by direct calculation that 
\begin{equation}\label{E:erelations}
[\ff_{r},\ff_{s}] \subseteq \ff_{r+s}
\end{equation} for all $r,s \in \Z_2.$
Thus $\ff = \ff_{\0} \oplus \ff_{\1}$ is a Lie subsuperalgebra of $\La$ which we call a \emph{detecting subalgebra} of $\fg .$  By \cite[Lemma 2.5]{LR} $\ff$ is unique up to conjugacy in the sense that if one chooses another semisimple $x \in \fg_{-1}\oplus \fg_{1}$ such that $\pi(x)$ is principal, then following the aforementioned 
construction leads to a detecting subalgebra which is $G_{0}$-conjuate to $\ff .$

\subsection{}\label{SS:detectingcohom}  
Applying the definition of relative cohomology in Section~\ref{SS:relcohom} one can calculate $\HH^{\bullet}(\ff, \ff_{\0};\C )$ as follows.
First, note that the $\Z_2$-grading implies that the differentials defining $\HH (\ff, \ff_{\0};\C )$ are identically zero (cf.\ the proof of \cite[Theorem 2.5.2]{BKN1}).  Thus the cohomology is given by the cochains; that is,
\[
\HH (\ff,\ff_{\0};\C ) \cong S\left(\ff_{\1}^{*} \right)^{\ff_{\0 }}.
\]  Furthermore, note that the elements of $S(\ff_{\1}^{*})$ which are invariant under $\ff_{\0 }$ are simply those of weight zero with respect to the torus $T.$  Therefore, recalling that $S(\ff_{\1}^{*}) \cong \C [\partial_{1}^{*}, Y_{2}, \dotsc , Y_{n}],$ one has 
\[
\HH (\ff,\ff_{\0};\C ) \cong S\left(\ff_{\1}^{*} \right)^{T}=\C [Y_{2}\partial_{1}^{*}, \dotsc , Y_{n}\partial_{1}^{*}].
\]

The following theorem relates the  $\fg$ and $\ff$ cohomology rings via the natural restriction map. As in the classical case \cite{BKN1}, one has that they are related via the invariants of a finite reflection group.

\begin{thm}\label{T:relatingcohom}  Let $\fg = W(n)$ and let $\ff$ be the detecting superalgebra defined above.  The inclusion map $\ff \hookrightarrow \fg$ induces a restriction map $\operatorname{res}: \HH^{\bullet}(\fg,\fg_0;\C ) \to \HH^{\bullet}(\ff ,\ff_{\0};\C )$ so that the following diagram commutes:

\begin{equation}\label{E:commute0}
\begin{CD}
   \HH^{\bullet} (\fg,\fg_{0};\C )  @>^{\operatorname{res}}>>       \HH^{\bullet} (\ff,\ff_{\0};\C )   \\
@V\cong VV                                                           @VV\cong V\\
  \C [Y_{2}\partial_{1}^{*}, \dotsc , Y_{n}\partial_{1}^{*}]^{\Sigma_{n-1}} @>^{\subseteq}>>      \C [Y_{2}\partial_{1}^{*}, \dotsc ,Y_{n}\partial_{1}^{*}]  
\end{CD}
\end{equation}  

 That is, the restriction map induced by inclusion gives the following graded algebra isomorphism,
\begin{equation*}
\HH (\fg,\fg_{0};\C ) \xrightarrow{\cong}  \HH (\ff,\ff_{\0 };\C )^{\Sigma_{n-1}} \cong \C [Y_{2}\partial_{1}^{*}, \dotsc , Y_{n}\partial_{1}^{*}]^{\Sigma_{n-1}}.\\
\end{equation*}
In particular, both $\HH (\fg,\fg_{0};\C )$ and $\HH (\ff,\ff_{\0 };\C )$ are isomorphic to graded polynomial rings in $n-1$ variables.  
\end{thm}

\begin{proof}  The isomorphisms are a reinterpretation of Theorem~\ref{T:calculatingR} in terms of cohomology and 
the commutativity of the diagram can be checked directly.  
\end{proof}

\section{Support varieties}\label{S:suppvar}

\subsection{Support Varieties}\label{SS:supportvarieties} Let $(\fa, \fa_{\diamond})$ denote one of the pairs 
$(\fg, \fg_{0})$, $(\ff , \ff_{\0}).$  Let  
$M$ and $N$ be objects of $\mathcal{C}:=\mathcal{C}_{(\fa,\fa_{\diamond})}$ such that
$$\Ext_{\mathcal{C}}^\bullet (M,N)\cong \HH^\bullet (\fa, \fa_{\diamond}; \Hom_{\C}(M, N))$$ is 
finitely generated as an $\HH (\fa,\fa_{\diamond};\C )$-module; e.g.\ when $M$ and $N$ are finite dimensional by Theorem~\ref{finitegeneration} for $(\fg, \fg_{0})$ and by \cite[Theorem 2.5.3]{BKN1} for $(\ff ,\ff_{\0 })$ (since $\ff$ is a classical Lie superalgebra). Let 
$$I_{(\fa, \fa_{\diamond})}(M, N)= \Ann_{\HH^{\bullet} (\fa,\fa_{\diamond};\C )}(\HH^\bullet (\fa , \fa_{\diamond}; \Hom_{\C}(M, N)))$$
be the annihilator ideal of this module. The  \emph{support variety} of the pair 
$(M,N)$ is defined to be
$$\V_{(\fa , \fa_{\diamond})}(M, N)= \MaxSpec(\HH^{\bullet} (\fa,\fa_{\diamond};\C )/I_{(\fa, \fa_{\diamond})}(M, N)),$$
the maximal ideal spectrum of the quotient of $\HH^{\bullet} (\fa,\fa_{\diamond};\C )$ by $I_{(\fa, \fa_{\diamond})}(M, N)$. 
In particular, when $M=N$, the \emph{support variety} of $M$ is 
$$\V_{(\fa, \fa_{\diamond})}(M)=\MaxSpec(\HH^{\bullet} (\fa,\fa_{\diamond};\C )/I_{(\fa , \fa_{\diamond})}(M, M)).$$  
In light of the calculations in the previous section $\V_{(\fa,\fa_{\diamond})}(M)$ can be identified with the conical affine subvariety of 
\[
\MaxSpec (\HH (\fa,\fa_{\diamond};\C ))=\V_{(\fa, \fa_{\diamond})}(\C)\cong \mathbb{A}^{n-1},
\]  defined by the ideal $I_{(\fa, \fa_{\diamond})}(M, M).$  For brevity in what follows we write 
\[
I_{(\fa, \fa_{\diamond})}(M)=I_{(\fa, \fa_{\diamond})}(M, M).
\]

\subsection{}\label{SS:relatingvarieties}  The inclusion $\ff \hookrightarrow \fg$ induces a restriction map on cohomology which, in turn, induces maps of support varieties.  That is, given supermodules $M$ and $N$ in $\mathcal{C}_{(\fg,\fg_{0})}$ one has $M, N \in \mathcal{C}_{(\ff,\ff_{\0})}$ by restriction to $\ff$ and one has maps of varieties
\begin{align*}
\resstar &:\V_{(\ff, \ff_{\0})}(M,N) \to \V_{(\fg, \fg_{0})}(M,N),\\
\resstar &:\V_{(\ff, \ff_{\0})}(M) \to \V_{(\fg, \fg_{0})}(M).
\end{align*}  Viewing the support varieties as subvarieties of $\mathbb{A}^{n-1}$ and using Theorem~\ref{T:relatingcohom} one can explicitly describe this map as the quotient by the action of $\Sigma_{n-1}$ on $\mathbb{A}^{n-1}$ by permutation of coordinates.  Therefore one has
\begin{equation}\label{E:resstarimage}
\V_{(\ff, \ff_{\0})}(M)/\Sigma_{n-1} \cong \resstar \left( \V_{(\ff, \ff_{\0})}(M)\right) \subseteq \V_{(\fg, \fg_{0})}(M).
\end{equation}  We conjecture that the inclusion in \eqref{E:resstarimage} is in fact an equality for all finite dimensional $\fg$-supermodules $M \in \mathcal{C}_{(\fg,\fg_{0})}$.

\subsection{Rank Varieties}\label{SS:rankvarieties} An important motivation for introducing the detecting subalgebra $\ff$ 
is that one can describe its support varieties using the theory of rank varieties.  However, the situation is 
markedly different than for the classical Lie superalgebras.  To demonstrate these differences and 
make the relationship precise we introduce another Lie subsuperalgebra of $\fg.$

Let $\tilde{\ff}_{\0}=\Lie (H) \subset \fg_{0}$ and $\tilde{\ff}_{\1}=\ff_{\1}.$  Set
\begin{equation}\label{E:tildefdef}
\tilde{\ff} = \tilde{\ff}_{\0} \oplus \tilde{\ff}_{\1}.
\end{equation}  As with $\ff$ one can verify that $\tilde{\ff}$ is a Lie subsuperalgebra of $\fg.$  In fact one has 
\begin{equation}\label{E:tildefbrackets}
\left[\tilde{\ff}_{\0},\tilde{\ff}_{\0} \right]= \left[\tilde{\ff}_{\0},\tilde{\ff}_{\1} \right]=0. 
\end{equation}
One can also verify that the differentials defining $\HH (\tilde{\ff},\tilde{\ff }_{\0};\C )$ are identically zero and so the cohomology ring is again given by the cochains.  In this case, however, $\tilde{\ff}_{\0}$ acts trivially on $\tilde{\ff}_{\1}$ and so one has 
\[
\HH (\tilde{\ff},\tilde{\ff}_{\0};\C ) \cong S(\tilde{\ff }_{\1}^{*} ) = S\left(\ff_{\1}^{*} \right) \cong \C [\partial_{1}^{*}, Y_{2}, \dotsc , Y_{n}].
\]  Furthermore, the inclusion $\tilde{\ff } \hookrightarrow \ff $ defines a restriction map, $\operatorname{res},$ so that the following diagram commutes,
\begin{equation}\label{E:commute}
\begin{CD}
   \HH^{\bullet}(\ff,\ff_{\0};\C )  @>^{\operatorname{res}}>>       \HH^{\bullet}(\tilde{\ff},\tilde{\ff}_{\0};\C )   \\
@V\cong VV                                                           @VV\cong V\\
   \C [\partial_{1}^{*}Y_{2}, \dotsc , \partial_{1}^{*}Y_{n}] @>^{\subseteq}>>     \C [\partial_{1}^{*}, Y_{2}, \dotsc , Y_{n}]  
\end{CD}
\end{equation}

By using the pair $(\fa, \fa_{\diamond})=(\tilde{\ff},\tilde{\ff }_{\0})$ and the setup of Section~\ref{SS:supportvarieties} one can define the support varieties $\V_{(\tilde{\ff},\tilde{\ff}_{\0})}(M,N)$ and $\V_{(\tilde{\ff},\tilde{\ff}_{\0})}(M)$ for any supermodules $M, N \in \mathcal{C}_{(\tilde{\ff },\tilde{\ff }_{\0})}$ for which $\Ext^{\bullet}_{\mathcal{C}_{(\tilde{\ff} , \tilde{\ff }_{\0})}}(M, N)$ and  $\Ext^{\bullet}_{\mathcal{C}_{(\tilde{\ff} , \tilde{\ff }_{\0})}}(M, M)$ are finitely generated $\HH^{\bullet}(\tilde{\ff} ,\tilde{\ff}_{\0};\C)$-modules.  Since $\HH (\tilde{\ff},\tilde{\ff }_{\0};\C )$ is a polynomial ring in $n$ variables, one can naturally identify these support varieties with the conical affine subvarieties of the affine $n$-space
\[
\operatorname{MaxSpec}\left( \HH (\tilde{\ff},\tilde{\ff }_{\0};\C ) \right) = \V_{(\tilde{\ff},\tilde{\ff}_{\0})}(\C ) \cong \mathbb{A}^{n}
\] defined by the ideals $I_{(\tilde{\ff},\tilde{\ff}_{\0})}(M,N)$ and $I_{(\tilde{\ff},\tilde{\ff}_{\0})}(M),$ respectively.

Alternatively, one can describe the $\tilde{\ff}$ support variety using rank varieties.  As a matter of notation, given a homogeneous 
element $x\in \tilde{\ff }$, let $\langle x \rangle$ denote the Lie subsuperalgebra generated by $x$.  If $M \in \mathcal{C}_{(\tilde{\ff},\tilde{\ff }_{\0})}$ is finite dimensional then
define the \emph{rank variety} of 
$M$ to be
$$\V_{\tilde{\ff}}^{\rank}(M)=\left\{x\in \tilde{\ff }_{\1}=\ff_{\1}\mid M \text{ is not projective as a $U\left(\langle x \rangle \right)$-supermodule} \right\}\cup \{0\}.$$

Since by \eqref{E:tildefbrackets} the structure of $\tilde{\ff}$ is of the type considered in \cite[Sections 5, 6]{BKN1}, \cite[Theorem 6.3.2]{BKN1} implies that one has a canonical isomorphism
\begin{equation}\label{E:rankiso}
\V _{\left(\tilde{\ff},\tilde{\ff}_{\0} \right)}(M) \cong \V_{\tilde{\ff}}^{\rank}(M)
\end{equation}
for any finite dimensional $\tilde{\ff}$-supermodule $M$ which is an object of $\mathcal{C}_{(\tilde{\ff},\tilde{\ff}_{\0})}.$ We identify the rank and support varieties of $\tilde{\ff}$ via this isomorphism.  

\subsection{Relating $\tilde{\ff }$ and $\ff$ Support Varieties}\label{SS:relatingfvarieties}  We now wish to relate the support varieties of $\tilde{\ff}$- and $\ff$-supermodules. Note that if $M \in \mathcal{C}_{(\ff ,\ff_{\0})},$ then via restriction it is an object in $\mathcal{C}_{(\tilde{\ff},\tilde{\ff }_{\0})}.$  Therefore, whenever $M$ is finite dimensional one has an induced map of varieties, 
\begin{equation*}
\resstar : \V _{\left(\tilde{\ff},\tilde{\ff}_{\0} \right)}(M) \to \V _{\left(\ff,\ff_{\0} \right)}(M).
\end{equation*}
The present task is to better understand this map.

As a conseqence of \eqref{E:commute} one has that the map 
\[
\resstar : \V _{\left(\tilde{\ff},\tilde{\ff}_{\0} \right)}(\C) \to \V _{\left(\ff,\ff_{\0} \right)}(\C).
\] is given by the canonical quotient map 
\[
 \V _{\left(\tilde{\ff},\tilde{\ff}_{\0} \right)}(\C) \to \V _{\left(\tilde{\ff},\tilde{\ff}_{\0} \right)}(\C)/T. 
\]  That is, for $M \in \mathcal{C}_{(\ff ,\ff_{0})}$ one has 
\begin{equation}\label{E:trivialresequality}
 \V _{\left(\tilde{\ff},\tilde{\ff}_{\0} \right)}(M)/T \cong \resstar \left( \V _{\left(\tilde{\ff},\tilde{\ff}_{\0} \right)}(M)  \right) \subseteq   \V _{\left(\ff,\ff_{\0} \right)}(\C).
\end{equation} In fact one has the following theorem.

\begin{thm}\label{T:generalresequality} Let $M$ be a finite dimensional object in $\mathcal{C}_{(\ff ,\ff_{\0})},$ then 
\begin{equation}\label{E:generalres}
\V _{\left(\tilde{\ff},\tilde{\ff}_{\0} \right)}(M)/T \cong \resstar \left(\V _{\left(\tilde{\ff},\tilde{\ff}_{\0} \right)}(M) \right)  = \V _{\left(\ff,\ff_{\0} \right)}(M).
\end{equation}
\end{thm}

\begin{proof} The first isomorphism is \eqref{E:trivialresequality}. It remains to show that the map $\resstar$ is surjective.

To do so first requires a better understanding of the relationship between $\tilde{\ff}$ and $\ff$ cohomology with coefficents in a finite dimensional $\ff$-supermodule $U$.  Recall the definition of the cochains for relative cohomology in Section~\ref{SS:relcohom} and that $\tilde{\ff}_{\1}=\ff_{\1}$.  If $U$ is a finite dimensional supermodule in $\mathcal{C}_{(\ff ,\ff_{\0})},$ then the torus $T$ acts on the cochains $C^{\bullet}(\tilde{\ff},\tilde{\ff}_{\0};U) = \Hom_{\tilde{\ff }_{\0}}(\Lambda^{\bullet}_{s}(\ff_{\1}), U)$ by $(t.\varphi)(x)=t\varphi(t^{-1}x)$ for all $t \in T,$ $\varphi \in C^{\bullet}(\tilde{\ff},\tilde{\ff}_{\0};U),$ and $x \in \Lambda_{s}^{\bullet}(\ff_{\1}).$  If $N$ is a $T$-module and $\lambda \in X(T)$ is a weight, then write $N_{\lambda}$ for the $\lambda$ weight space of $N$.  Since $T$ acts semisimply on $C^{\bullet}(\tilde{\ff},\tilde{\ff}_{\0};U)$ one has
\[
C^{\bullet}(\tilde{\ff},\tilde{\ff}_{\0};U) = C^{\bullet}(\tilde{\ff},\tilde{\ff}_{\0};U)^{T} \oplus \bigoplus_{\substack{\lambda \in X(T) \\ \lambda \neq 0}}  C^{\bullet}(\tilde{\ff},\tilde{\ff}_{\0};U)_{\lambda} 
\] as $T$-modules.  Observe that the action of $T$ commutes with the differential in the definition of relative cohomology.  Thus one has 
\begin{align*}
\HH (\tilde{\ff},\tilde{\ff}_{\0}; U) &= \HH (\tilde{\ff},\tilde{\ff}_{\0}; U)^{T} \oplus  \bigoplus_{\substack{\lambda \in X(T) \\ \lambda \neq 0}} \HH (\tilde{\ff},\tilde{\ff}_{\0}; U)_{\lambda}  \\
 & \cong \HH (\ff,\ff_{\0}; U) \oplus  \bigoplus_{\substack{\lambda \in X(T) \\ \lambda \neq 0}} \HH (\tilde{\ff},\tilde{\ff}_{\0}; U)_{\lambda} , \\
\end{align*} where the isomorphism follows from the equality  $C^{\bullet}(\tilde{\ff},\tilde{\ff}_{\0};U)^{T}=C^{\bullet}(\ff,\ff_{\0};U)$ and the exactness of taking $T$ invariants.  In particular, one has 
\begin{equation}\label{E:embedgeneralcohom}
\operatorname{res} : \HH (\ff,\ff_{\0}; U) \xrightarrow{\cong} \HH (\tilde{\ff},\tilde{\ff}_{\0}; U)^{T} \subseteq \HH (\tilde{\ff},\tilde{\ff}_{\0}; U).
\end{equation}

We are now prepared to prove the theorem.  Let $(\fa ,\fa_{\0})$ denote either $(\ff ,\ff_{\0})$ or $(\tilde{\ff},\tilde{\ff}_{\0}).$  Note that, just as for finite groups, an equivalent characterization of $I_{(\fa ,\fa_{\0})}(M)$ is the ideal of elements in $\HH^{\bullet}(\fa ,\fa_{\0};\C )$ which annihilate the element $1_{\fa, M} \in \Ext^{0}_{\mathcal{C}_{(\fa ,\fa_{\0})}}(M,M)$ corresponding to the identity morphism.  Note, too, that $\operatorname{res}(1_{\ff ,M})=1_{\tilde{\ff},M}$ for any $\ff$-supermodule $M$ and that $\operatorname{res}(x.z)=\operatorname{res}(x).\operatorname{res}(z)$ for any $x \in \HH^{\bullet}(\ff ,\ff_{\0};\C )$ and $z \in \Ext^{\bullet}_{\mathcal{C}_{(\ff ,\ff_{\0})}}(M,M)$.

Since the ideal $\operatorname{res}^{-1}\left( I_{(\tilde{\ff } ,\tilde{\ff}_{\0})}(M) \right)$ defines the variety $\resstar \left( V_{(\tilde{\ff } ,\tilde{\ff}_{\0})}(M)\right),$  it suffices to prove 
\[
\operatorname{res}^{-1}\left(  I_{(\tilde{\ff } ,\tilde{\ff}_{\0})}(M)\right) =  I_{(\ff  ,\ff_{\0})}(M).
\]

Let $x \in I_{(\ff  ,\ff_{\0})}(M).$  That is, $x.1_{\ff, M}=0$ and so 
\[
0=\operatorname{res}(x.1_{\ff, M})= \operatorname{res}(x).\operatorname{res}(1_{\ff ,M})=\operatorname{res}(x).1_{\tilde{\ff},M}.
\]  That is, $\operatorname{res}(x) \in I_{(\tilde{\ff } ,\tilde{\ff}_{\0})}(M) $ and so $ x \in \operatorname{res}^{-1}\left(  I_{(\tilde{\ff } ,\tilde{\ff}_{\0})}(M)\right).$

Conversely, let $x \in \operatorname{res}^{-1}\left(  I_{(\tilde{\ff } ,\tilde{\ff}_{\0})}(M)\right).$ Then 
\[
0=\operatorname{res}(x).1_{\tilde{\ff},M}=\operatorname{res}(x).\operatorname{res}(1_{\ff,M})=\operatorname{res}\left(x.1_{\ff,M} \right).
\]  However by \eqref{E:embedgeneralcohom} (applied to the case $U=M^{*}\otimes M$) one has that $\operatorname{res}$ is injective and so $0=x.1_{\ff ,M}.$  That is, $x \in  I_{(\ff  ,\ff_{\0})}(M). $  This proves \eqref{E:generalres}.
\end{proof}


\subsection{Properties of $\ff$ Support Varieties}\label{SS:propsoffvareties}
 We record some basic properties of support varieties for $\ff$-supermodules which follow from the rank variety description of $\tilde{\ff}$ support varieties and the isomophism given in Theorem~\ref{T:generalresequality}.  The situation is reminiscent of the connection between support varieties for the Frobenius kernels $G_{r}$ and $G_{r}T$ considered in \cite{nakano}. Other properties of rank varieties can be found in \cite[Theorem 6.4.2]{BKN1}.

\begin{thm}\label{T:fvarietyprops}Let $M,N,M_{1},M_{2}$ and $M_{3}$ be finite dimensional $\ff$-supermodules in $\mathcal{C}_{(\ff, \ff_{\0})}$. Then,
\begin{itemize}
\item [(a)] $\V_{(\ff,\ff_{\0})}(M\otimes N)=\V_{(\ff,\ff_{\0} )}(M)\cap \V_{(\ff,\ff_{\0} )}(N)$;
\item [(b)] $\V_{(\ff ,\ff_{\0})}(M^\ast)=\V_{(\ff ,\ff_{\0})}(M);$
\item [(c)] $\V_{(\ff ,\ff_{\0})}(M^\ast\otimes M)=\V_{(\ff ,\ff_{\0})}(M);$
\item [(d)] If 
\[
0 \to M_{1} \to M_{2} \to M_{3} \to 0
\] is a short exact sequence, then 
\[
\V_{(\ff ,\ff_{\0})}(M_{i}) \subseteq \V_{(\ff ,\ff_{\0})}(M_{j}) \cup \V_{(\ff ,\ff_{\0})}(M_{k}),
\] where $\{i,j,k \}=\{1,2,3 \}.$
\end{itemize}
\end{thm}

\begin{proof} We first make the following observations.  Consider a finitely generated, commutative, graded algebra $S$ where each graded summand is finite dimensional.  Furthermore, assume some group, $\Gamma,$ acts semisimply on $S$ and the action respects the grading on $S.$  Let $\iota: S^{\Gamma} \hookrightarrow S$ be the canonical embedding.  If $J$ is a graded ideal of $S,$ then the ideal $\iota^{-1}\left(J \right) = J^{\Gamma}.$  In addition, if $I,J$ are both graded ideals of $S,$ then one has 
\begin{equation}\label{E:idealinvariants}
\left(I + J \right)^{\Gamma} = I^{\Gamma} + J^{\Gamma}.
\end{equation}  Namely, one first notes that one has the inclusion $I^{\Gamma} + J^{\Gamma} \subseteq \left(I + J \right)^{\Gamma}.$  However, as graded $\Gamma$-modules one has 
\[
\left(I + J \right)/\left(I \cap J \right) \cong I \oplus J.
\] By using the fact that taking fixed points under $\Gamma$ is 
exact (because the action of $\Gamma$ is semisimple) one has  
\[
\left(I + J \right)^{\Gamma}/\left(I \cap J \right)^{\Gamma} \cong \left(I \oplus J \right)^{\Gamma}.
\]  However, $\left(I \cap J \right)^{\Gamma}= I^{\Gamma} \cap J^{\Gamma}$ and $\left(I \oplus J \right)^{\Gamma} = I^{\Gamma} \oplus J^{\Gamma}.$  Thus one has 
\begin{equation}\label{E:isoA}
\left(I + J \right)^{\Gamma}/\left(I^{\Gamma} \cap J^{\Gamma} \right) \cong \left(I^{\Gamma} \oplus J^{\Gamma} \right).
\end{equation}
On the other hand, considering $I^{\Gamma}$ and $J^{\Gamma}$ as $\Gamma$-modules one has 
\begin{equation}\label{E:isoB}
\left(I^{\Gamma}+ J^{\Gamma} \right)/\left(I^{\Gamma} \cap J^{\Gamma} \right) \cong I^{\Gamma} \oplus J^{\Gamma}.
\end{equation}  Using \eqref{E:isoA} and \eqref{E:isoB} to compare dimensions of the graded summands of \eqref{E:idealinvariants}, one sees that the earlier inclusion must, in fact, be an equality.

To prove part (a), one first notes that as a consequence of the rank variety description on has by \cite[Proposition 6.3.1]{BKN1} that 
\[
 \V_{(\tilde{\ff },\tilde{\ff }_{\0})}(M\otimes N)=\V_{(\tilde{\ff },\tilde{\ff}_{\0} )}(M)\cap \V_{(\tilde{\ff},\tilde{\ff}_{\0} )}(N).
\]  As a matter of notation, if $J$ is an ideal, let $\sqrt{J}$ denote its radical ideal.  Then the above equality translates into the equality 
\[
\sqrt{I_{(\tilde{\ff},\tilde{\ff}_{\0})}(M\otimes N)} =\sqrt{I_{(\tilde{\ff},\tilde{\ff}_{\0})}(M) + I_{(\tilde{\ff},\tilde{\ff}_{\0})}(N)}.
\] Taking invariants with respect to $T$ and applying  \eqref{E:idealinvariants}, one obtains 
\begin{align*}
\sqrt{I_{(\ff,\ff_{\0 })}(M\otimes N)} &=\sqrt{I_{(\tilde{\ff},\tilde{\ff}_{\0 })}(M\otimes N)^{T}} \\
&= \left(\sqrt{I_{(\tilde{\ff},\tilde{\ff}_{\0 })}(M\otimes N)} \right)^{T} \\
& =\left( \sqrt{I_{(\tilde{\ff},\tilde{\ff}_{\0 })}(M) + I_{(\tilde{\ff},\tilde{\ff}_{\0 })}(N)}\right)^{T}\\
& =\sqrt{I_{(\tilde{\ff},\tilde{\ff}_{\0 })}(M)^{T} + I_{(\tilde{\ff},\tilde{\ff}_{\0})}(N)^{T}}\\
& =\sqrt{I_{(\ff,\ff_{\0 })}(M) + I_{(\ff,\ff_{\0})}(N)}.
\end{align*}  This proves the desired equality of varieties.

Part (b) is proven by a similar but easier argument and part (c) follows from parts (a) and (b).  

Finally, to prove part (d) one observes that the rank variety description implies (cf.  \cite[Theorem 6.4.2(d)]{BKN1}) that one has 
\[
\V_{(\tilde{\ff} ,\tilde{\ff}_{\0})}(M_{i}) \subseteq \V_{(\tilde{\ff} ,\tilde{\ff }_{\0})}(M_{j}) \cup \V_{(\tilde{\ff} ,\tilde{\ff }_{\0})}(M_{k}),
\]  where $\{i,j,k \}=\{1,2,3 \}.$  One then argues as in part (a) using instead that 
\[
\left( \sqrt{I_{(\tilde{\ff},\tilde{\ff}_{\0})}(M_{j})I_{(\tilde{\ff},\tilde{\ff}_{\0})}(M_{k})}\right)^{T}= \sqrt{I_{(\tilde{\ff},\tilde{\ff}_{\0})}(M_{j})^{T}I_{(\tilde{\ff},\tilde{\ff}_{\0})}(M_{k})^{T}}=\sqrt{I_{(\ff,\ff_{\0})}(M_{j})I_{(\ff,\ff_{\0})}(M_{k})}.
\]
\end{proof}

Another important property of support varieties is their ability to detect projectivity.  This is illustrated by the following theorem.

\begin{thm}\label{T:detectingprojectives}  Let $M$ be a finite dimensional supermodule in $\mathcal{C}_{(\ff ,\ff_{\0})}.$ Then the following are equivalent:
\begin{enumerate}
\item [(a)] The supermodule $M$ is projective in $\mathcal{C}_{(\ff ,\ff_{\0})};$
\item [(b)] The supermodule $M$ is projective in $\mathcal{C}_{(\tilde{\ff} ,\tilde{\ff }_{\0})};$
\item [(c)] The variety $\V_{(\tilde{\ff } ,\tilde{\ff}_{\0})}(M) = \{0 \}.$

\end{enumerate}
\end{thm}

\begin{proof}  If $M$ is a projective $\ff$-supermodule, then it remains so upon restriction to $\tilde{\ff},$ hence one has that (a) implies $(b).$

To prove (b) implies (a) it suffices to show 
\[
\Ext^{i}_{\mathcal{C}_{(\ff,\ff_{\0})}}(M,L)\cong \HH^{i}(\ff ,\ff_{\0 }; M^{*}\otimes L)=0
\]
for all objects $L$ in $\mathcal{C}_{(\ff ,\ff_{\0})}$ and $i >0.$  Since $\tilde{\ff}$ is an ideal in $\ff$ one can consider the Lyndon-Hochschild-Serre spectral sequence for the pairs $(\tilde{\ff},\tilde{\ff}_{\0}) \subseteq (\ff,\ff_{\0})$:
\[
E_{2}^{i,j}=\HH^{i}\left( \ff /\tilde{\ff},\ff_{\0}/\tilde{\ff}_{\0}; \HH^{j}(\tilde{\ff },\tilde{\ff}_{\0};M^{*}\otimes L)\right)\Rightarrow   
\HH^{i+j}(\ff,\ff_{\0};M^{*}\otimes L).
\] 
By assumption $M$ is a projective object in $\mathcal{C}_{(\tilde{\ff},\tilde{\ff}_{\0})}$ and so $\HH^{j}(\tilde{\ff },\tilde{\ff}_{\0};M^{*}\otimes L)=0$ for $j > 0$ and the spectral sequence collapses.  That is, for $i \geq 0$ one has 
\[
\HH^{i}\left( \ff /\tilde{\ff},\ff_{\0}/\tilde{\ff}_{\0};(M^{*}\otimes L)^{\tilde{\ff}}\right) \cong   
\HH^{i}(\ff,\ff_{\0}; M^{*}\otimes L).
\]  Since the objects of $\mathcal{C}_{(\ff,\ff_{\0})}$ are finitely semisimple as $\ff_{\0}$-supermodules and since $\ff /\tilde{\ff} = \ff_{\0}/\tilde{\ff}_{\0}$ is a one dimensional subtorus of $\ff_{\0}$ one has $\HH^{i}(\ff /\tilde{\ff},\ff_{\0}/\tilde{\ff}_{\0}; (M^{*} \otimes L)^{\tilde{\ff}}) = 0$ for $i > 0.$  It follows that $\HH^{i}(\ff,\ff_{\0}; M^{*}\otimes L)=0$ and so $M$ is projective in $\mathcal{C}_{(\ff ,\ff_{\0})}$.

The equivalence of (b) and (c) follows from \cite[Theorem 6.4.2(b)]{BKN1}.
\end{proof}

Note that it is \emph{not} true that if $\V_{(\ff ,\ff_{\0})}(M)=\{0 \},$ then $M$ is projective as a $\ff$-supermodule.  One can find examples of $\ff$-supermodules, $M,$ so that $\V_{(\tilde{\ff},\tilde{\ff}_{\0})}(M) \neq \{0 \},$ but by \eqref{E:generalres} 
$$  \V_{(\ff, \ff_{\0})}(M) \cong \V_{(\tilde{\ff}, \tilde{\ff}_{\0})}(M)/T=\{0 \}.$$  
On the other hand, by the previous theorem $M$ is not projective as an $\ff$-supermodule since $\V_{(\tilde{\ff},\tilde{\ff}_{\0})}(M) \neq \{0 \}.$


\section{Calculation of Support Varieties} 

\subsection{Calculation of Support Varieties for Kac Supermodules}\label{SS:Kacsupports}  Recall that for $\lambda\in X_{0}^{+}$ we constructed the Kac supermodule $K(\lambda)$. The following 
result shows that the $\fg$ and $\ff$ support varieties are zero for all Kac supermodules.

\begin{prop} \label{P:kacsupport} Let $\la \in X_{0}^{+}$ and $N$ be a finite dimensional supermodule in ${\mathcal C}_{(\fg,
\fg_{0})}$. Then,
\begin{itemize}
\item [(a)] $\V_{(\La, \La_0)}(K(\la),N)=\{0\};$
\item [(b)] $\V_{(\La, \La_0)}(K(\la))=\{0\};$
\item [(c)]  $\V_{(\ff, \ff_{\0})}(K(\la))=\{0\}.$
\end{itemize}
\end{prop}
\begin{proof} We present a modified version of the argument used to prove \cite[Theorem 3.2.1]{BKN2}. 
First observe that part (b) follows immediately from part (a). Also, as in the proof of \cite[Corollary~3.3.1]{BKN2}, for part (a) it suffices to 
prove that for $n$ sufficiently large, $\Ext^{n}_{\mathcal{C}_{(\fg,\fg_{0})}}(K(\lambda), N)=0$. 

By Frobenius reciprocity, for all $n$ we have 
$$\Ext^{n}_{\mathcal{C}_{(\La,\La_0)}}(K(\lambda),N)\cong \Ext^{n}_{\mathcal{C}_{(\fg_{0}\oplus 
\fg^{+},\La_0)}}( L_{0}(\lambda),N).$$ 
Since $\fg^{+}$ is an ideal in $\fg_{0}\oplus\fg^{+}$ one can apply the Lyndon-Hochschild-Serre spectral sequence to $(\fg^{+}, \{0 \}) \subseteq (\fg_{0}\oplus \fg^{+},\fg_{0})$: 
$$E_{2}^{i,j}=\Ext^{i}_{\mathcal{C}_{(\fg_0,\fg_0)}}(L_{0}(\lambda),
\Ext^{j}_{\mathcal{C}_{(\fg^{+},\{0\})}}({\mathbb C},N))\Rightarrow   
\Ext^{i+j}_{\mathcal{C}_{(\fg_{0}\oplus \fg^{+},\La_0)}}(L_{0}(\lambda),N).$$

Since ${\mathcal{C}_{(\fg_{0},\fg_{0})}}$ consists of $\fg_{0}$-supermodules which are finitely semisimple over $\fg_{0},$ this spectral sequence is zero for $i >0.$ That is, it collapses 
at the $E_{2}$ page and yields 
\begin{equation}\label{E:iso}
\Hom_{\fg_{0}}(L_{0}(\lambda), \Ext^{n}_{\mathcal{C}_{(\fg^{+},\{0\})}}({\mathbb C},N)) 
\cong \Ext^{n}_{\mathcal{C}_{(\fg_{0}\oplus \fg^{+},\La_0)}}(L_{0}(\lambda),N). 
\end{equation} 
According to the definition of relative cohomology, $\Ext^{n}_{\mathcal{C}_{(\fg^{+},\{0\})}}({\mathbb C},N)$ 
is a subquotient of $\Lambda_{s}^{n}((\fg^{+})^{*}) \otimes N$. But 
$\Lambda_{s}^{n}((\fg^{+})^{*})$ is positively graded by degree and $N$ is finite dimensional 
so for sufficiently large $n$ (depending on $\lambda$), 
$\Lambda_{s}^{n}((\fg^{+})^{*}) \otimes N$ contains no composition factors of the form 
$L_{0}(\lambda)$. Thus $\Ext^{n}_{\mathcal{C}_{(\fg,\fg_{0})}}(K(\lambda), N)=0$ for $n \gg 0$.

Finally, to prove part (c) it suffices to observe that the map $\resstar :\V_{(\ff, \ff)_{\0}}(K(\la)) \to \V_{(\La, \La_0)}(K(\la))=\{0 \}$ is finite-to-one.  Since $\V_{(\ff, \ff_{\0})}(K(\la))$ is a conical variety it follows that it must be 
equal to $\{0\}$. 
\end{proof}

\subsection{} Recall that Serganova \cite[Lemma 5.3]{Ser} proved that the set of  atypical weights for $\La$ is
$$\Omega =\{a\varepsilon_i+\varepsilon_{i+1}+ \dotsb  +\varepsilon_n\mid a\in \C,\ 1\leq i\leq n\}.$$
Moreover, she determined the characters of the simple $\fg$-supermodules by 
determining composition series for the Kac supermodules. Serganova's abridged results for finite dimensional 
simple supermodules are presented in the following theorem.  

\begin{thm}\label{thm:supp} \cite[Theorem 6.3, Corollary 7.6]{Ser} Let $\lambda\in X_{0}^{+}$. 
\begin{itemize} 
\item[(a)] If $\lambda \notin \Omega$ then $K(\lambda)\cong L(\lambda)$. 
\item[(b)] Let $\lambda \in \Omega .$
\begin{itemize} 
\item[(i)] If $\la=a\varepsilon _i+\varepsilon _{i+1}+\dotsb +\varepsilon _n$ with $a\neq 0,1$, then there is the following exact sequence:
\begin{equation}\label{E:supp1}0\rightarrow L(\la-\varepsilon _i)\rightarrow K(\la)\rightarrow L(\la)\rightarrow 0.
\end{equation} 
\item[(ii)] The structure of $K(0)$ and $K(\varepsilon _1+ \dotsb  +\varepsilon _n)$ is described by 
the exact sequences
\begin{equation}\label{E:supp2}
0\rightarrow L(-\varepsilon _n)\rightarrow K(0)\rightarrow L(0)\rightarrow 0,
\end{equation}
\begin{equation}\label{E:supp3}0\rightarrow L(0)\rightarrow K(\varepsilon _1+ \dotsb  +\varepsilon _n)\rightarrow L(\varepsilon _1+ \dotsb  +\varepsilon _n)\rightarrow 0.
\end{equation}
\end{itemize} 
\end{itemize} 
\end{thm} 

From the above theorem one has an alternative characterization of typical/atypical for $\lambda\in X_{0}^{+}:$ namely, $\lambda$ is typical if and only if $K(\lambda)$ is simple. 

\subsection{Calculation of Support Varieties for Simple Supermodules} The following theorem presents the computation of support varieties of simple $\fg$-supermodules. 
Our results demonstrate that $L(\lambda)$ is typical if and only if the support variety of $L(\lambda)$ is zero.  

\begin{thm} Let $\la \in X_{0}^{+}$ and let $L(\la)$ be finite dimensional simple $\La$-supermodule
with highest weight $\la$. Then,
\begin{itemize} 
\item[(a)] If $\la \notin \Omega$ then $\V_{(\La, \La_0)}(L(\la))=\V_{(\ff,\ff_{\0})}(L(\la))=\{0\};$
\item[(b)] If $\la \in \Omega$ then $\V_{(\ff,\ff_{\0})}(L(\la))=\V_{(\ff,\ff_{\0})}(\C)$ and $\V_{(\La, \La_0)}(L(\la))=\V_{(\La, \La_0)}(\C)$.  
\end{itemize}
\end{thm}
\begin{proof} Part (a) is immediate from Proposition~\ref{P:kacsupport}(c) and Theorem ~\ref{thm:supp}(a). 

To prove part (b), first observe that it suffices to prove $\V_{(\ff,\ff_{\0})}(L(\la))=\V_{(\ff,\ff_{\0})}(\C).$ Namely, one will then have
\[
  \V_{(\fg, \fg_{0})}(\C) = \resstar \left( \V_{(\ff, \ff_{\0})}(\C )\right)=\resstar \left( \V_{(\ff, \ff_{\0})}(L(\lambda))\right) \subseteq \V_{(\La, \La_0)}(L(\la)) \subseteq  \V_{(\La, \La_0)}(\C),
\] which implies the result for $\fg.$
Furthermore, recall from Section~\ref{S:OnWn} that 
$$\Omega\cap X_{0}^{+}=\{a\varepsilon_{1}+\dots +\varepsilon _n \mid a = 1,2,3,\dotsc \}\cup 
\{b\varepsilon _n \mid\ b=0,-1,\dots \}.$$

We will repeatedly use two facts about support varieties of $\ff$-supermodules: the support variety of a supermodule in a short exact sequence is 
contained in the union of support variety of the other two supermodules by Theorem~\ref{T:fvarietyprops}(d), and the support variety of a Kac supermodule is zero by Proposition~\ref{P:kacsupport}(c). 

Note that $L(0)\cong \C$ and set $\V:=\V_{(\ff,\ff_{\0})}(\C)$. From \eqref{E:supp2} it follows that 
$\V=\V_{(\ff,\ff_{\0})}(L(-\epsilon_{n}))$. One can now use the remarks from the previous paragraph and \eqref{E:supp1} to recursively prove that $\V=\V_{(\ff,\ff_{\0})}(L(-b\epsilon_{n}))$ 
for $b=-2,-3,\dots$. 
Similarly, we have $\V=\V_{(\ff,\ff_{\0})}(L(\epsilon_{1}+\dots+\epsilon_{n}))$ 
from \eqref{E:supp3}.  Applying \eqref{E:supp1} recursively shows that 
$\V=\V_{(\ff,\ff_{\0})}(L(a\epsilon_{1}+\dots+\epsilon_{n}))$ for all $a = 2, 3, \dotsc $.  Note that all elements of $\Omega \cap X_{0}^{+}$ were considered above and thus the theorem is proven.
\end{proof}

\section{Realization of Support varieties}\label{S:realization}

\subsection{} One important property in the theory of support varieties is the realizability 
of any conical variety as the support variety of some module in the category. Carlson \cite{Ca} first 
proved this for finite groups in the 1980s. Friedlander and Parshall \cite{FPa}  
later used Carlson's proof to establish realizability for restricted 
Lie algebras. For arbitrary finite group schemes the finite 
generation of cohomology due to Friedlander and Suslin \cite{FS} allowed one to 
define support varieties. In this generality the realizability of supports 
was established using Friedlander and Pevtsova's method \cite{FPe} of concretely 
describing support varieties through $\pi$-points. 

In the classical Lie superalgebra setting the realizability of supports 
was established for the detecting subalgebra $\fe$ in \cite[Theorem 6.4.3]{BKN1}. 
The main tool to establish this theorem is the tensor product 
theorem \cite[Proposition 6.3.1]{BKN1}. Since the detecting subalgebra for $W(n)$ also has the tensor product theorem by Theorem~\ref{T:fvarietyprops}(a), it follows by the argument used in the classical case that the realization theorem also holds for $\ff$. The goal of this section is 
to lift these realization theorems to the support varieties of $\fg$ 
where $\fg$ is a classical Lie superalgebras as considered in \cite{BKN1} or $W(n)$.

\subsection{} There is a slight difference in the way that the support varieties and detecting subalgebras are defined in the two cases.  We will fix a common notation which allows us to treat both cases more or less simultaneously.  Let $\fg$ be a classical Lie superalgebra with a polar and stable 
action of $G_{\0}$ on $\fg_{\bar{1}}$ as in \cite{BKN1} or let $\fg=W(n)$. Let $\mathcal{C}=\mathcal{C}_{(\fg ,\fg_{\0})}$ if $\fg$ is classical and  $\mathcal{C}=\mathcal{C}_{(\fg ,\fg_{0})}$ if $\fg = W(n).$
Let 
$$
\HH^{\bullet}= 
\begin{cases}
\HH^{\bullet}(\fg,\fg_{\0};{\mathbb C}), & \text{if $\fg$ is classical;}\\
\HH^{\bullet}(\fg,\fg_{0};{\mathbb C}),  & \text{if $\fg=W(n)$.}
\end{cases}
$$
Let $\fa$ denote the detecting subalgebra for $\fg$ and let $W$ denote the finite pseduoreflection group such that 
\[
\operatorname{res}: \HH^{\bullet} \xrightarrow{\cong} \HH^{\bullet}(\fa ,\fa_{\0};\C )^{W} \subseteq \HH^{\bullet}(\fa ,\fa_{\0};\C ).
\]  
If $M$ and $N$ are objects in $\mathcal{C}$ for which $\Ext_{{\mathcal C}}^{\bullet}(M,N)$ (resp.\ 
$\Ext_{{\mathcal C}}^{\bullet}(M,M)$) is a finitely generated $\HH^{\bullet}$-module, then write $\V_{\fg}(M,N)$ for the corresponding relative support variety (resp.\ $\V_{\fg}(M)$ for the corresponding support variety). 

 As a matter of notation, if $J$ is an ideal of some commutative ring $A,$ then let $\mathcal{Z}(J)$ be the variety defined by $J.$  That is,
\[
\mathcal{Z}(J) = \left\{\mathfrak{m} \in \MaxSpec (A) \mid J \subseteq \mathfrak{m}  \right\}.
\]  In particular, for $a \in A$ let $\mathcal{Z}(a)$ denote the variety defined by the ideal $(a).$

\subsection{Tensor Products}\label{SS:tensorproduct} One of the fundamental results in the theory of support varieties for finite group schemes is that the support variety of the tensor product of two modules is the intersection of the two modules' support varieties.  Lacking equality in \eqref{E:resstarimage} we are limited to the following analogue.

\begin{lem} \label{L:tensorproduct}  Let $\fg$ be a classical, stable, and polar Lie superalgebra or let $\fg$ be $W(n).$  Let $\fa$ be the detecting subalgebra of $\fg$.  Let $M$ and $N$ be $\fa$-supermodules in $\mathcal{C}_{(\fa ,\fa_{\0})}$ for which 
$\Ext^{\bullet}_{\mathcal{C}_{(\fa ,\fa_{\0})}}(M, M),$ $\Ext^{\bullet}_{\mathcal{C}_{(\fa ,\fa_{\0})}}(N,N),$ and 
$\Ext^{\bullet}_{\mathcal{C}_{(\fa ,\fa_{\0})}}(M\otimes N, M\otimes N)$ are finitely generated $\HH^{\bullet}(\fa, \fa_{\0};\C )$-modules. Then,
\[
\resstar \left(\V_{(\fa,\fa_{\0})}\left(M \otimes N \right) \right) = \resstar \left(\V_{(\fa,\fa_{\0})}\left(M \right) \right) \cap \resstar \left(\V_{(\fa,\fa_{\0})}\left(N \right) \right).
\]

\end{lem}

\begin{proof}  If $M$ is a finite dimensional $\fg$-supermodule in $\mathcal{C}$ and $I_{(\fa ,\fa_{\0})}(M)$ is the ideal which defines $\V_{(\fa ,\fa_{\0})}(M)$, then $\resstar \left(\V_{(\fa ,\fa_{\0})}(M) \right)$ is the variety defined by the ideal $\operatorname{res}^{-1}\left(I_{(\fa ,\fa_{\0})}(M) \right).$  Recall that 
\[
\operatorname{res}: \HH^{\bullet} \xrightarrow{\cong} \HH^{\bullet}(\fa,\fa_{\0};\C )^{W} \subseteq \HH^{\bullet}(\fa,\fa_{\0};\C ),
\] where $W$ is a finite group.  If we identify $\HH^{\bullet}$ with its image under this map, one has $\operatorname{res}^{-1}(J)=J^{W}$ for any ideal $J$ in $ \HH^{\bullet}(\fa,\fa_{\0};\C ).$  Furthermore, given an ideal $I$ we write $\sqrt{I}$ for the radical of the ideal.

By the tensor product property of $\fa$ support varieties (cf. \cite[Proposition 6.3.1(a)]{BKN1} and Theorem~\ref{T:fvarietyprops}(a)) one has 
\[
\resstar \left(\V_{(\fa,\fa_{\0})}\left(M \otimes N \right) \right) =\resstar \left(\V_{(\fa,\fa_{\0})}\left(M \right)  \cap  \V_{(\fa,\fa_{\0})}\left(N \right) \right). 
\]  Applying the earlier remarks, at the level of ideals the above equality becomes 
\[
\sqrt{I_{(\fa ,\fa_{\0})}(M \otimes N)^{W}} = \sqrt{\left(I_{(\fa ,\fa_{\0})}(M) + I_{(\fa ,\fa_{\0})}(N)\right)^{W}}.
\]  However by \eqref{E:idealinvariants} one has 
\[
\sqrt{\left(I_{(\fa ,\fa_{\0})}(M) + I_{(\fa ,\fa_{\0})}(N)\right)^{W}} = \sqrt{I_{(\fa ,\fa_{\0})}(M)^{W} + I_{(\fa ,\fa_{\0})}(N)^{W}}.
\]  As the latter ideal defines the variety $\resstar \left(\V_{(\fa,\fa_{\0})}\left(M \right) \right) \cap  \resstar \left( \V_{(\fa,\fa_{\0})}\left(N \right) \right), $ this yields the desired result.
\end{proof}

\subsection{Carlson Supermodules} To prove realizability one needs to introduce a family of supermodules for which one can explicitly calculate their support varieties.  Let $n >0$ and let $\zeta \in \HH^{n}$. We can consider $\zeta$ to be a $\fg$-homomorphism from 
the $n$th syzygy of the trivial supermodule, $\Omega^{n}({\mathbb C}),$ to ${\mathbb C}$. Set 
$$L_\zeta =\Ker(\zeta : \Omega ^n(\C)\rightarrow \C)\subseteq \Omega ^n(\C).$$
These supermodules are often referred to as ``Carlson modules.''  As in the theory of support varieties 
for finite group schemes the importance of the supermodule $L_\zeta $ is that one can explicitly 
realize its support as the zero locus of $\zeta$ in $\text{MaxSpec}(\HH^{\bullet})$.  

\subsection{}\label{SS:CarlsonSupportsforDetecting}  The first step is to compute the support variety of $L_{\zeta}$ over the detecting subalgebra.

\begin{lem}\label{L:Lzetavariety}  Let $\fg$ be a classical, stable, and polar Lie superalgebra or let $\fg$ be $W(n)$.  Let $\fa$ denote the detecting subalgebra of $\fg.$ Given  $\zeta \in \HH^n$ and $L_{\zeta}$ as above, then 
\begin{equation*}
\V_{(\fa ,\fa_{\0})}(L_{\zeta}) =\V_{(\fa ,\fa_{\0})}(L_{\operatorname{res}(\zeta)}) = \mathcal{Z}(\operatorname{res}(\zeta)).
\end{equation*}

\end{lem}
  
\begin{proof} Given  $\zeta \in \HH^n$, we compute the $\fa$ support variety of $L_\zeta $ as follows. Construct the short exact sequence of $\fg$-supermodules, 
$$0\rightarrow L_\zeta \rightarrow \Omega ^n(\C)\xrightarrow{\zeta} \C\rightarrow 0.$$ 
Upon restriction to $\fa$ one obtains the short exact sequence,
$$0\rightarrow L_\zeta \downarrow _{\fa}\rightarrow \Omega ^n(\C)\downarrow _\fa\xrightarrow{\operatorname{res}(\zeta)} \C\rightarrow 0.$$

By using the graded version of Schaunel's Lemma, $L_\zeta \downarrow _{\fa}\cong L_{\operatorname{res}(\zeta )}\oplus P $ 
and $\Omega ^n(\C)\downarrow _\fa\cong \Omega_\fa^n(\C)\oplus P$, where $\Omega_\fa^n(\C)$
denotes $\Omega^n(\C)$ for the trivial $\fa$-supermodule, $L_{\operatorname{res}(\zeta)}$ is the Carlson $\fa$-supermodule for $\operatorname{res}(\zeta)\in \HH^{n}(\fa,\fa_{\0};\C ),$ and $P$ is some projective $\fa$-supermodule. 
Therefore, we have the following short exact sequence of $\fa$-supermodules:

$$0\rightarrow L_{\operatorname{res}(\zeta )}\oplus P\rightarrow \Omega_\fa^n(\C)\oplus P\rightarrow \C\rightarrow 0. $$
By the rank variety description (cf.\ \cite[Theorem 6.4.3]{BKN1}) of $\V_{(\fa,\fa_{\0})}(L_{\operatorname{res}(\zeta)})$ when $\fg$ is classical or the rank variety description of $\V_{(\tilde{\ff},\tilde{\ff}_{\0})}(L_{\operatorname{res}(\zeta)})$  and Theorem~\ref{T:generalresequality} when $\fg =W(n),$ one has that $\V_{(\fa ,\fa_{\0})}(L_{\operatorname{res}(\zeta)}) = \mathcal{Z}(\operatorname{res}(\zeta)).$  Therefore, since $L_\zeta \downarrow _{\fa}\cong L_{\operatorname{res}(\zeta )}\oplus P,$ one has that  
\begin{equation*}
\V_{(\fa ,\fa_{\0})}(L_{\zeta}) = \V_{(\fa ,\fa_{\0})}(L_{\operatorname{res}(\zeta)}) = \mathcal{Z}(\operatorname{res}(\zeta)).
\end{equation*}  
\end{proof}

\subsection{}\label{SS:CarlsonmodsforW} We should warn the reader that if $\fg =W(n)$ it may be that $L_{\zeta}$ is infinite dimensional and, hence, $\Ext^{\bullet}_{\mathcal{C}}(L_{\zeta},L_{\zeta})$ is no longer necessarily finitely generated as an $\HH^{\bullet}$-module.  Let us mention that since as $\fa$-supermodules $L_{\zeta} \cong L_{\operatorname{res}(\zeta)} \oplus P$ and  $L_{\operatorname{res}(\zeta)}$ is finite dimensional (since the projective indecomposible $\fa$-supermodules are finite dimensional by \cite[Proposition 5.2.2]{BKN2}), this complication did not arise in Lemma~\ref{L:Lzetavariety}. Similarily, when $\fg$ is classical $L_{\zeta}$ is necessarily finite dimensional.  However, if one wishes to consider support varieties for $W(n)$, then the issue can no longer be ignored.  To circumvent this difficulty one can instead choose to work with relative support varieties as we now demonstrate.

\begin{prop}\label{P:relativeext} Let $\fg=W(n)$ and let $\zeta _1,\dots, \zeta _s \in \HH^{\bullet}$ be 
homogeneous elements with corresponding Carlson supermodules 
$L_{\zeta_1}, \dots, L_{\zeta_s}$. 
\begin{itemize} 
\item[(a)] Then $\operatorname{Ext}^{\bullet}_{\mathcal{C}}(L_{\zeta_1}\otimes 
\dotsb  \otimes L_{\zeta_s},{\mathbb C})$ is finitely generated over $\HH^{\bullet}$. 
\item[(b)] ${\mathcal V}_\fg(L_{\zeta_1}\otimes 
\dotsb \otimes L_{\zeta_s},{\mathbb C})\subseteq \cap_{i=1}^{s} {\mathcal V}_\fg(L_{\zeta_i},\C)$
\end{itemize} 
\end{prop} 

\begin{proof} (a) We will prove this by induction on $s$. For $s=1$, consider the 
short exact sequence 
\begin{equation}\label{E:anotherSES}
0\rightarrow L_{\zeta}\xrightarrow{\alpha} \Omega ^n(\C)\xrightarrow{\zeta} \C \rightarrow 0.
\end{equation}

This induces a long exact sequence of $\HH^{\bullet}$-modules: 
\[\dotsb 
\xrightarrow{d} \Ext^{t}_{\mathcal{C}}(\C,\C) \xrightarrow{\zeta_{*}} 
\Ext^{t}_{\mathcal{C}}(\Omega^{n}(\C),\C)\xrightarrow{\alpha_{*}} 
\Ext^{t}_{\mathcal{C}}(L_{\zeta},\C)\xrightarrow{d} \Ext^{t+1}_{\mathcal{C}}\left(\C ,\C  \right) \to \dotsb,  
\] where $\alpha_{*}$ and $\zeta_{*}$ are the maps induced by $\alpha$ and $\zeta,$ respectively, and $d$ denotes the connecting morphism in the long exact sequence. 

For $t \geq 0,$ set 
\begin{align*}
A_{t} &= \Ext^{t}_{\mathcal{C}}(\Omega^{n}(\C),\C)/ \Ker \left(\alpha_{*} \right), \\
B_{t} &= \operatorname{Im} \left(d \right) \subseteq  \Ext^{t+1}_{\mathcal{C}}\left(\C ,\C  \right).
\end{align*}  Let $A_{\bullet}= \oplus_{t} A_{t}$ and $B_{\bullet} = \oplus_{t} B_{t}.$  Note that  
\begin{align*}
\alpha_{*}^{\bullet}&:\Ext^{\bullet}_{\mathcal{C}}(\Omega^{n}(\C),\C)\to \Ext^{\bullet}_{\mathcal{C}}(L_{\zeta},\C) \\
d^{\bullet}:&\Ext^{\bullet}_{\mathcal{C}}(L_{\zeta},\C)\xrightarrow{d} \Ext^{\bullet}_{\mathcal{C}}\left(\C ,\C  \right)
\end{align*} are $\HH^{\bullet}$-module homomorphisms, where $\alpha_{*}^{\bullet}$ and $d^{\bullet}$ are the maps obtained by taking the direct sum of the maps $\alpha_{*}$ and $d,$ respectively.  Hence, $A_{\bullet}$ and $B_{\bullet}$ are $\HH^{\bullet}$-modules and from the long exact sequence given above one has the short exact seqence of $\HH^{\bullet}$-modules, 
\begin{equation}\label{E:SES}
0 \to A_{\bullet} \xrightarrow{\bar{\alpha}_{*}^{\bullet}} \Ext^{\bullet}_{\mathcal{C}}(L_{\zeta},\C) \xrightarrow{\bar{d}^{\bullet}} B_{\bullet} \to 0,
\end{equation} where $\bar{\alpha}_{*}^{\bullet}$ and $\bar{d}^{\bullet}$ are the maps induced by $\alpha_{*}^{\bullet}$ and $d^{\bullet},$ respectively.

However, for all $t \geq 0,$ $\Ext^{t}_{\mathcal{C}}(\Omega^{n}(\C),\C)\cong \Ext^{n+t}_{\mathcal{C}}(\C,\C)$ by degree shifting.  Taking the direct sum of these maps yields an $\HH^{\bullet}$-module isomorphism 
\[
\Ext^{\bullet}_{\mathcal{C}}(\Omega^{n}(\C),\C)\cong \Ext^{\bullet}_{\mathcal{C}}(\C,\C).
\] Therefore $\Ext^{\bullet}_{\mathcal{C}}(\Omega^{n}(\C),\C)$
is finitely generated over $\HH^{\bullet}.$  Since $\HH^{\bullet}$ is a Noetherian ring it follows that the quotient module $A_{\bullet}$ is finitely generated over $\HH^{\bullet}.$  Similarily, since $\Ext^{\bullet}_{\mathcal{C}}(\C,\C)$ is finitely generated over $\HH^{\bullet},$ the submodule $B_{\bullet}$ is also finitely generated. Finally, using \eqref{E:SES} and the fact that $\HH^{\bullet}$ is Noetherian one has that  $\Ext^{\bullet}_{\mathcal{C}}(L_{\zeta},\C)$ must be finitely generated over $\HH^{\bullet}$. 

For the inductive step we claim that if $M$ is a module in $\mathcal{C}$ 
for which $\Ext^{\bullet}_{\mathcal{C}}(M, \C)$ is a finitely 
generated $\HH^{\bullet}$-module and $\zeta$ is a homogeous element of $\HH^{\bullet},$ then $\Ext^{\bullet}_{\C}(M\otimes L_{\zeta}, \C)$ 
is finitely generated over $\HH^{\bullet}$. The argument parallels the base case considered above.  Namely, consider the short exact sequence obtained by tensoring \eqref{E:anotherSES} with $M:$
$$0\rightarrow M\otimes L_{\zeta} \rightarrow
M\otimes \Omega^{n}(\C )  \rightarrow
M\otimes \C 
\rightarrow 0.$$ 
Note that by assumption (i) $\Ext^{\bullet}_{\mathcal{C}}
(M,\C)$ is finitely generated over $\HH^{\bullet}$, and
(ii) $\Ext^{\bullet}_{\mathcal{C}}(M\otimes\Omega^{n}(\C),\C)$ is 
finitely generated over $\HH^{\bullet}$ by applying a dimension shift argument (as in the $s=1$ case). 
Applying the long exact sequence in cohomology and arguing as in the base case shows that 
$\Ext^{\bullet}_{\mathcal{C}}
(M\otimes L_{\zeta},\C)$ is finitely generated over $\HH^{\bullet}$. 

(b) The statement clearly holds for $s=1$. Now assume that the statement holds for $s-1$ factors. 
For a fixed $i=1,2,\dots,s$, set $N=L_{\zeta_{1}}\otimes \dotsb   \otimes \widehat{L_{\zeta_{i}}}\otimes \dotsb  \otimes L_{\zeta_{s}}$. 
Consider the following short exact sequence given by $\zeta_{i}$, 
$$0\rightarrow L_{\zeta_{i}}\rightarrow \Omega ^n(\C)\rightarrow \C \rightarrow 0.$$ 
By tensoring by $N$ we obtain 
$$0\rightarrow L_{\zeta_{i}}\otimes N\rightarrow \Omega ^n(\C)\otimes N\rightarrow N \rightarrow 0.$$  
Therefore, by induction
\begin{align*}
\V_\fg(L_{\zeta_1}\otimes 
\dotsb  \otimes L_{\zeta_s},{\mathbb C})&\subseteq \V_\fg(\Omega^{n}(\C)\otimes N,\C)\cup \V_\fg(N,\C)\\
     &= \V_{\fg}(N,\C ) \\
     &\subseteq \V_\fg(L_{\zeta_1},\C)\cap \dotsb  \cap \widehat{\V_\fg(L_{\zeta_i},\C)}\cap \dotsb  \cap \V_\fg(L_{\zeta_s},\C). 
\end{align*}
Since $i$ is arbitrary we conclude that 
$$\V_\fg(L_{\zeta_1}\otimes \dotsb  \otimes L_{\zeta_s},{\mathbb C}) \subseteq 
\cap_{i=1}^{s} \V_\fg(L_{\zeta_i},\C).$$
\end{proof} 

\subsection{Support Varieties for Carlson Supermodules}\label{SS:SupportsforCarlson}  We are now prepared to compute (relative) support varieties for the Carlson supermodules.

\begin{prop}\label{P:Lzetavarietyforg}  Let  $\zeta \in \HH^n$ and $L_{\zeta}$ be given as above. 

If $\fg$ is a classical, stable, and polar Lie superalgebra, then one has
\begin{equation}\label{E:classicalCarlson}
\V_{\fg}(L_{\zeta}) = \resstar \left( \V_{(\fa ,\fa_{\0})}(L_{\zeta})\right) = \mathcal{Z}(\zeta).
\end{equation}

If $\fg =W(n),$ then one has 
\begin{equation}\label{E:WCarlson}
\V_{\fg}(L_{\zeta},{\mathbb C}) = \resstar \left( \V_{(\fa ,\fa_{\0})}(L_{\zeta})\right) = \mathcal{Z}(\zeta).
\end{equation}

\end{prop}
  
\begin{proof} We first prove \eqref{E:classicalCarlson}.  Since 
\[
\resstar : \V_{(\fa, \fa_{\0})}(L_\zeta )\rightarrow \V_{\La}(L_\zeta ),
\] one has that
\[
\resstar \left( \V_{(\fa, \fa_{\0})}(L_\zeta ) \right) \subseteq \V_{\La}(L_\zeta ). 
\]  However, by Lemma~\ref{L:Lzetavariety} the variety $\resstar \left( \V_{(\fa, \fa_{\0})}(L_\zeta ) \right)$ is defined by the ideal $\operatorname{res}^{-1}\left((\operatorname{res}(\zeta) )\right) = (\zeta)$, where the equality of ideals follows from the explicit description of the map $\operatorname{res}.$  Therefore, one has 
\begin{equation}\label{E:Z1classical}
\mathcal{Z}(\zeta)  =\resstar \left( \V_{(\fa, \fa_{\0})}(L_\zeta ) \right) \subseteq  \V_{\La}(L_\zeta ). 
\end{equation}

On the other hand, one can use the proof given in  \cite[Proposition 6.13]{Ca2}  to show 
that $\zeta^2$ annihilates $\Ext^\bullet_{\mathcal{C}}(L_\zeta ,L_\zeta )$.  Let $I_\fg(L_\zeta)$ denote the annihilator of $\HH^{\bullet}$ acting on this $\Ext$ group (i.e.\ the ideal which defines the support variety). So we have 
$\zeta^2 \in I_{\La}(L_\zeta)$. This implies that $I_{\La}(L_\zeta )$ contains the ideal generated by $\zeta^2$.
This in turn implies that the radical of the ideal $I_{\La}(L_\zeta)$ contains the ideal generated by $\zeta $.
Thus the variety defined by the ideal $\sqrt{I_{\La}(L_\zeta)}$ is contained in $\mathcal{Z}(\zeta)$, that is,
\begin{equation}\label{E:Z2classical}\V_{\La}(L_\zeta)\subseteq \mathcal{Z}(\zeta).
\end{equation}
Combining equations \eqref{E:Z1classical} and \eqref{E:Z2classical} one has \eqref{E:classicalCarlson}.

To prove \eqref{E:WCarlson} one argues much as above.  Namely, one has
 \[
\resstar : \V_{(\ff, \ff_{\0})}(L_\zeta,\C )\rightarrow \V_{\La}(L_\zeta,\C ).
\] However, recall that $L_{\zeta} \cong L_{\operatorname{res}(\zeta)}\oplus P$ as $\ff$-supermodules where $P$ is a projective $\ff$-supermodule.  Also note that by definition $L_{\operatorname{res}(\zeta)}$ can be assumed to lie within the principal block of $\ff.$  However, by \cite[Proposition 5.2.2]{BKN2} the trivial supermodule is the only simple supermodule in the principal block of $\ff.$  Taken together with \cite[Proposition 5.7.1]{Ben2} these observations imply that
$\V_{(\ff,\ff_{\0})}\left(L_{\zeta} \right)=\V_{(\ff,\ff_{\0})}
\left(L_{\zeta},{\mathbb C} \right)$. Then Lemma~\ref{L:Lzetavariety} implies that the variety 
$\resstar \left( \V_{(\ff, \ff_{\0})}(L_\zeta ) \right)$ is defined by the ideal 
$\operatorname{res}^{-1}\left((\operatorname{res}(\zeta) )\right) = (\zeta)$.  Therefore, one has 
\begin{equation}\label{E:Z1}
\mathcal{Z}(\zeta)  =\resstar \left( \V_{(\ff, \ff_{\0})}(L_\zeta,\C ) \right) \subseteq  
\V_{\La}(L_\zeta,\C ). 
\end{equation}

On the other hand, one can once again use the proof given in  \cite[Proposition 6.13]{Ca2} to show 
that $\zeta^2$ annihilates $\Ext^\bullet_{\mathcal{C}}(L_\zeta ,L_{\zeta})$ when $\fg=W(n)$.  As for finite groups (cf.\ \cite[Section 5.7]{Ben2}), this implies $\zeta^{2}$ annihilates $\Ext^\bullet_{\mathcal{C}}(L_\zeta ,\C)$ and so is an element of the ideal which defines $\V_{\fg}(L_{\zeta},\C ).$  
Just as before this implies
\begin{equation}\label{E:Z2}\V_{\La}(L_\zeta,\C)\subseteq \mathcal{Z}(\zeta).
\end{equation}
Combining equations \eqref{E:Z1} and \eqref{E:Z2} one has \eqref{E:WCarlson}.
\end{proof}

\subsection{Realization Theorem} We are now ready to prove the realization theorem for $\fg$ being either a classical, stable, and polar Lie superalgebra as in \cite{BKN1} or $W(n).$ 
 
\begin{thm} Let $\fg$ be a classical, stable, polar Lie superalgebra or let $\fg = W(n).$  Let $X$ be a conical subvariety of $\V_{\fg}(\C ).$  If $\fg$ is classical, then there exists a finite
dimensional supermodule $M$ in $\mathcal{C}$ such that
$$ \V_{\La}(M)=X.$$
If $\fg =W(n),$ then there exists a supermodule $M$ in $\mathcal{C}$ such that 
$$ \V_{\La}(M,\C)=X.$$
\end{thm}
\begin{proof} First, express $X$ as the zero locus of homogeneous elements $\zeta _1, \dotsc , \zeta _s \in \HH^{\bullet}$. That is, fix homogeneous elements  $\zeta _1, \dotsc , \zeta _n \in \HH^\bullet$ such that 
$$X=\mathcal{Z}(\zeta _1)\cap  \dotsb \cap \mathcal{Z}(\zeta _s).$$
Let $M=L_{\zeta _1}\otimes \dotsb \otimes L_{\zeta _s}$. If $\fg$ is classical then one can combine \eqref{E:classicalCarlson}, Lemma~\ref{L:tensorproduct}, and the fact that $\V_{\fg}(N_{1} \otimes N_{2}) \subseteq \V_{\fg}(N_{1}) \cap \V_{\fg}(N_{2})$ for any two supermodules $N_{1},N_{2}$ in $\mathcal{C}$ (cf. \cite[(4.6.4)]{BKN2}) to obtain
\begin{align*}
X &=\cap_{i=1}^{s} \mathcal{Z}(\zeta_{i}) \\
  &=\cap_{i=1}^{s}\V_{\fg}\left(L_{\zeta_{i}} \right) \\
  &= \cap_{i=1}^{s} \resstar \left(\V_{(\fa ,\fa_{\0})}\left(L_{\zeta_{i}} \right) \right) \\
  &= \resstar \left(\V_{(\fa,\fa_{\0})}(M) \right) \\
  & \subseteq \V_{\fg }\left(M \right) \\
  & \subseteq \cap_{i=1}^{s} \V_{\fg}\left(L_{\zeta_{i}} \right) \\
  &  = \cap_{i=1}^{s} \mathcal{Z}(\zeta_{i}) \\
  & = X.
\end{align*}  It then follows that $\V_{\fg}(M)=X.$

The case when $\fg =W(n)$ is argued similarily using instead \eqref{E:WCarlson} and Proposition~\ref{P:relativeext}(b).
\end{proof}

\subsection{Representation Type} Germoni \cite{Ger} investigated the representation type of the Lie superalgebra 
$\mathfrak{sl}(m|n)$. He proved that if $m,n\geq 2$ then $\mathfrak{sl}(m|n)$ 
has wild representation type. Germoni also conjectured that this should hold for 
blocks of atypicality greater than or equal to two. Later, Shomron \cite{Sho} proved 
that each block of the Lie superalgebra $W(n)$ has wild representation type for $n\geq 3$. 
Both cases are based on studying the $\Ext^{1}$ quiver. 

Recently Farnsteiner \cite[Theorem 3.1]{Far} showed that if the dimension of the support variety of any simple module 
in a  block for a finite group scheme has dimension 
at least three, then the block has wild representation type. The proof chiefly depends upon using the finite group scheme 
analogue of the above realizability result to construct sufficiently many indecomposable modules in the block.  
With these results in mind we present the following conjecture relating the representation type 
of Lie superalgebras with our construction of support varieties for both the 
classical and Cartan type Lie superalgebras. 

\begin{conj}\label{C:reptypeconj} Let $\mathcal{B}$ be a block of $\mathcal{C}$.  If there exists a simple supermodule $S$ in $\mathcal{B}$ 
with $\dim \V_\fg(S)\geq 3,$ then $\mathcal{B}$ has wild representation type. 
\end{conj} 

In light of Conjecture~\ref{C:reptypeconj} and Germoni's conjecture on the representation type of the blocks of $\mathfrak{sl}(m|n),$ it is worthwhile to note that by the 
calculations in \cite{BKN2} one has that if $\mathcal{B}$ is a block of $\mathfrak{gl}(m|n)$ 
of atypicality $k,$ then 
\[
\V_{\fg}(S) \cong \mathbb{A}^{k}
\] for all simple supermodules $S$ in ${\mathcal B}$.


\begin{thebibliography}{9999999}\frenchspacing

\bibitem[\textsf{Ben1}]{Ben1}
D.~J. Benson, {\em Representations and Cohomology. {I}}, second ed., Cambridge
  Studies in Advanced Mathematics, vol.~30, Cambridge University Press,
  Cambridge, 1998.

\bibitem[\textsf{Ben2}]{Ben2}
\bysame, {\em Representations and Cohomology. {II}}, second ed., Cambridge
  Studies in Advanced Mathematics, vol.~31, Cambridge University Press,
  Cambridge, 1998.

\bibitem[\textsf{BKN1}]{BKN1}
B.~D. Boe, J.~R. Kujawa, and D.~K. Nakano, {\em Cohomology and
  support varieties for {L}ie superalgebras}, ar{X}iv:math.RT/0609363, (2006).

\bibitem[\textsf{BKN2}]{BKN2}
\bysame, {\em Cohomology and
  support varieties for Lie superalgebras II}, Proc. London Math. Soc., to appear.

\bibitem[\textsf{Ca1}]{Ca} J.F. Carlson, {\em The variety of an indecomposable module is connected}, 
Invent. Math., \textbf{77} (1984), 291--299.

\bibitem[\textsf{Ca2}]{Ca2} \bysame, \emph{Modules and group algebras}, Notes by Ruedi Suter, Lectures in Mathematics ETH Zurich. Birkhäuser Verlag, Basel, 1996. 

\bibitem[\textsf{CM}]{CM} D. Collingwood and W. McGovern, \emph{Nilpotent Orbits in Semisimple Lie Algebras}, 
Van Nostrand Reinhold Mathematics Series. Van Nostrand Reinhold Co., New York, 1993. 

\bibitem[\textsf{DK}]{dadokkac}
J. Dadok and V. Kac, {\em Polar representations}, J. Algebra \textbf{92}
  (1985), no.~2, 504--524.

\bibitem[\textsf{DS}]{DS} M. Duflo and V. Serganova, \emph{On associated variety for Lie superalgebras}, ar{X}iv:math.RT/0507198, (2005).


\bibitem[\textsf{Far}]{Far} R. Farnsteiner, {\em Tameness and complexity of finite group schemes}, 
Bull. London Math. Soc. \textbf{39} (2007), no.~1, 63--70.

\bibitem[\textsf{FPa}]{FPa} E.M. Friedlander, B.J. Parshall, {\em Support varieties 
for restricted Lie algebras}, {Invent. Math.}, \textbf{86} (1986), 553--562. 

\bibitem[\textsf{FPe}]{FPe} E.M. Friedlander, J. Pevtsova, {\em 
Representation theoretic support spaces for finite group schemes}, {American J. Math.}, 
\textbf{127} (2005), 379--420. 

\bibitem[\textsf{FS}]{FS} E.M. Friedlander, A. Suslin, {\em Cohomology of finite 
group schemes over a field}, {Invent. Math.}, \textbf{127}  
(1997), no. 2, 209--270.  

\bibitem[\textsf{Ger}]{Ger} J. Germoni, {\em Indecomposable representations of special linear 
Lie superalgebras}, J. Algebra, \textbf{209} (1998), 367--401. 

\bibitem[\textsf{Hum}]{Hum} J. E. Humphreys, \emph{Conjugacy Classes in Semisimple Algebraic Groups}, 
Mathematical Surveys and Monographs, 43. American Mathematical Society, Providence, RI, 1995. 

\bibitem[\textsf{Jan}]{Jan} J.C. Jantzen, {\em Representations of  
Algebraic Groups}, Second Edition, American Mathematical Society, 
Providence R.I., 2003. 

\bibitem[\textsf{Kac}]{Kac}
V.~G. Kac, {\em Lie superalgebras}, Advances in Math. \textbf{26} (1977),
  no.~1, 8--96.

\bibitem[\textsf{KW}]{kacwakimoto}
V.~G. Kac and M. Wakimoto, {\em Integrable highest weight modules over
  affine superalgebras and number theory}, Lie theory and geometry, Progr.
  Math., vol. 123, Birkh\"auser Boston, Boston, MA, 1994, 415--456.

\bibitem[\textsf{Kum}]{Kum} S. Kumar, {\em Kac-Moody Groups, their Flag Varieties and Representation Theory}, 
Progress in Mathematics, vol.204, 2002. 

\bibitem[\textsf{LR}]{LR}
D.~Luna and R.~W. Richardson, {\em A generalization of the {C}hevalley
  restriction theorem}, Duke Math. J. \textbf{46} (1979), no.~3, 487--496.

\bibitem[\textsf{Nak}]{nakano} D. K. Nakano,  \emph{Varieties for $G\sb rT$-modules},  Group representations: cohomology, group actions and topology (Seattle, WA, 1996), Proc. Sympos. Pure Math., \textbf{63}, Amer. Math. Soc., (1998), 441--452.

\bibitem[\textsf{PV}]{PV} V. L. Popov and E. B. Vinberg, {\em Invariant Theory}, Algebraic Geometry IV 
(A. N. Parshin and I. R.Shafarevich, eds.), Encylopedia of Mathematical Sciences, vol. 55, Springer-Verlag, Berlin, 1994.

\bibitem[\textsf{Pr}]{Pr} A.A. Premet, \emph{The theorem on restriction of invariants and nilpotent elements 
in $W_{n}$}, Math. USSR Sbornik, \textbf{73} (1992), 135--159. 

\bibitem[\textsf{Sch}]{Sch} M. Scheunert, \emph{The Theory of Lie Superalgebras}, Lec. Notes Math (Springer-Verlag) \textbf{716}, 1976. 

\bibitem[\textsf{Ser}]{Ser} V. Serganova, \emph{On representations of Cartan type Lie superalgebras},
Amer. Math. Soc. Transl., {\bf 213}  (2005), no.~2, 223--239. 

\bibitem[\textsf{Sho}]{Sho} N. Shomron, \emph{Blocks of Lie superalgebras of type $W(n)$}, J. Algebra 
\textbf{251} (2002), 739--750. 


\end{thebibliography}
\end{document}